\documentclass[11pt, oneside]{article}   	
\usepackage{geometry}                		
\usepackage{graphicx}				
								
\usepackage{epstopdf}
\usepackage{placeins}
\DeclareGraphicsExtensions{.eps}	
\usepackage{amssymb,amsmath,amsthm}
\usepackage{algorithmic}
\usepackage{caption}
\usepackage{subcaption}
\usepackage{color}
\usepackage{centernot}
\usepackage{bbm}

\newtheorem{theorem}{Theorem}
\newtheorem{lemma}[theorem]{Lemma}

\newtheorem{corollary}[theorem]{Corollary}
\newtheorem{proposition}[theorem]{Proposition}

\newtheorem{remark}[theorem]{Remark}

\newcommand{\var}{\operatorname{Var}}
\newcommand{\cov}{\operatorname{Cov}}
\newcommand{\Hm}{\operatorname{Hm}}

\newcommand{\mc}[1]{\mathcal{#1}}

\newcommand{\mbb}[1]{\mathbb{#1}}
\newcommand{\til}[1]{\tilde{#1}}

\newcommand{\tB}{{\tilde{B}}}

\newcommand{\E}{{\mathbb{E}}}

\newcommand{\ctG}{{\tilde{\mathcal{G}}}}

\newcommand{\cI}{{\mathcal{I}}}

\newcommand{\cK}{{\mathcal{K}}}

\newcommand{\PP}{{\mathbb{P}}}

\newcommand{\ZZ}{{\mathbb{Z}}}

\newcommand{\one}{\mathbf{\mathbbm{1}}}
\newcommand{\0}{\mathbf{0}}
\newcommand\numberthis{\addtocounter{equation}{1}\tag{\theequation}}

\usepackage{url}

\newcommand*\samethanks[1][\value{footnote}]{\footnotemark[#1]}

\title{
Percolation for level-sets of Gaussian free fields on metric graphs
}
\author{ Jian Ding\thanks{Partially supported by an NSF grant DMS-1757479, an Alfred Sloan fellowship.}  \\ University of Pennsylvania \and Mateo Wirth\samethanks \\ University of Pennsylvania
}
\date{\today}

\begin{document}

\thispagestyle{empty}
\maketitle
\begin{abstract}
We study level-set percolation for Gaussian free fields on metric graphs. In two dimensions, we give an upper bound on the chemical distance between the two boundaries of a macroscopic annulus. Our bound holds with high probability conditioned on connectivity and is sharp up to a poly-logarithmic factor with an exponent of one-quarter. This substantially improves a previous result by Li and the first author. In three dimensions and higher, we provide rather precise estimates of percolation probabilities in different regimes which altogether describe a sharp phase transition.
\end{abstract}



\linespread{1.2}


\section{Introduction}
	
\subsection{Gaussian free fields on metric graphs}
	
	In this paper, we study Gaussian free fields on metric graphs of integer lattices, which are closely related to (discrete) Gaussian free fields on integer lattices. We begin with some basic definitions before stating our main results. Let $\{S_t : t \geq 0\}$ be a continuous-time random walk on $\ZZ^d$ with transition rates $\frac{1}{2d}$. For $d\geq 3$, the (discrete) Gaussian free field on $\mathbb Z^d$, $\{\phi_v: v\in \mathbb Z^d\}$, is defined as a mean-zero Gaussian process with covariance $\E (\phi_u \phi_v)$ given by  (denoting below by $\one_A$ the indicator function of the event $A$)
	\begin{equation}\label{Green definition}
	G(u,v) = \E_u\left[ \int_0^{\infty} \one_{S_t = v} dt\right] \quad u,v \in \mathbb Z^d\,.
	\end{equation}
	It is clear that the preceding definition cannot extend to $d=2$ because simple random walk is recurrent in the two-dimensional lattice. For this reason (as usual), for $d=2$ we define the Gaussian free field on a finite set $V\subset \mathbb Z^2$ with Dirichlet boundary conditions, denoted by $\{\phi_v: v\in V\}$, to be a mean zero Gaussian process with covariance $\E (\phi_u \phi_v)$ given by  
	\begin{equation}\label{Green definition}
	G(u,v) = \E_u\left[ \int_0^{\zeta} \one_{S_t = v} dt\right] \quad u,v \in V,
	\end{equation}
	where $\zeta = \inf\{t \geq 0 \,:\, S_t \in \partial V\}$ is the hitting time of the internal boundary $\partial V = \{v \in V\,:\, \exists u \in V^c,\, |u-v| = 1\}$.
	
	Let $\mathcal G = \mathcal G(V, E)$ be the subgraph of $\mathbb Z^d$ on $V$, where we usually let $V$  be a finite box for $d=2$ and we take $V = \mathbb Z^d$ for $d \geq 3$. To each $e\in E$ we associate a different compact interval $I_e$ of length $d$ and identify the endpoints of this interval with the two vertices adjacent to $e$. The metric graph $\ctG$ associated to $V$ is then defined to be $\ctG = \cup_{e\in E} I_e$. With this definition, it was shown in \cite{Lupu16} that the Gaussian free field on $\ctG$, denoted by $\{\til{\phi}_v \,:\, v \in \ctG\}$, can be constructed in two equivalent ways. The first is by extending $\phi$ to $\ctG$ in the following manner: for adjacent vertices $u,v$, the value of $\til{\phi}$ on the edge $e(u,v)$, conditioned on $\phi_u$ and $\phi_v$, is given by an independent bridge of length $d$ of a Brownian motion with variance 2 at time 1. We note in passing that we have chosen the convention that each edge of $\mc{G}$ has conductance $\frac{1}{2d}$ in order to be consistent with \cite{LawlerLimic10}.
	
	Alternatively, one can construct $\til{\phi}$ by first defining a Brownian motion $\{\til{B}_t \,:\, t \geq 0\}$ on $\ctG$ as in \cite[Section 2]{Lupu16}. $\til{B}$ behaves like a standard Brownian motion in the interior of the edges, while on the vertices (i.e. lattice points) it chooses to do excursions on each incoming edge uniformly at random (see \cite{Lupu16} for further details). We let $\til\zeta = \inf\{ t \geq 0 \,:\, \til{B}_t \in \partial V\}$ for $d=2$, $\til\zeta = \infty$ for $d\geq 3$, and by an abuse of notation let $\{G(u,v) \,:\, u,v \in \ctG\}$ be the density of the 0-potential of $\{\til{B}_t \,:\, 0 \leq t < \til\zeta\}$ (with respect to the Lebesgue measure on $\ctG$), where $u$ and $v$ are now arbitrary points in $\til{\mc{G}}$ (not necessarily  vertices). It is shown in \cite{Lupu16} that the trace of $\til{B}$ on $V$ (when parametrized by its local time at the vertices) is exactly the continuous-time simple random walk on $V$ (killed at $\partial V$ for $d=2$), and that therefore the two definitions of $G$ coincide for $u,v \in V$, justifying the abuse of notation. The Gaussian free field $\{\til{\phi}_v \,:\, v \in \ctG\}$ is then the continuous, mean-zero Gaussian field on $\ctG$ with covariance given by $\E[\til{\phi}_u \til{\phi}_v] = G(u,v)$. It was also shown in \cite{Lupu16} that the value of $G$ on the edges of $\til{\mc{G}}$ can be obtained by interpolation from the value on the vertices. For two pairs of adjacent vertices $(u_1, v_1)$ and $(u_2, v_2)$ in $V$, and two points $w_1 \in e(u_1, v_1)$ and $w_2 \in e(u_2, v_2)$ on the corresponding edges, taking the convention that either the edges are distinct or $(u_1,v_1) = (u_2,v_2)$ and letting $r_1 = |w_1 - u_1|$ and $r_2 = |w_2 - u_2|$ (here we are measuring the standard Euclidean distance), we have (c.f. \cite[Equation (2.1)]{Lupu16})
	\begin{align*}
	G(w_1, w_2) = &(1- r_1)(1-r_2) G(u_1,u_2) + r_1 r_2 G(v_1, v_2) + (1-r_1)r_2G(u_1,v_2) \\
	&+ r_1(1-r_2)G(v_1, u_2) + 2d(r_1\wedge r_2 - r_1r_2) \one_{(u_1,v_1) = (u_2, v_2)}. \numberthis \label{Metric Green Interpolation}
	\end{align*}

	\subsection{Main results}
	
	The main goal of the present paper is to study level-set percolation for Gaussian free fields on metric graphs. For $r \geq 1$ we let $V_r = [-r,r]^d \cap \ZZ^d$ be the points in the latice contained in the box of side-length $2r$ centered at the origin (we choose this convention so that all boxes can be centered at the origin). For $d = 2$ we take a sequence $\til{\phi}_N$ of fields defined on the metric graphs $\ctG_N$ associated to $V_N$ (with Dirichlet boundary conditions). For $h \in \mathbb R$ we let $\til{E}_{N}^{\geq h} = \{v \in \ctG_N \,:\, \til{\phi}_{N, v} \geq h\}$ be the level set, or excursion set, of $\til{\phi}_N$ above $h$ --- note that our choice of level set is different from that of \cite{DingLi18} by a flipping symmetry, in order to be consistent with the majority of the literature. Further, for $u,v \in \til{E}_{N}^{\geq h}$, we let the chemical distance $D_{N, h}(u,v)$ be the graph distance between $u$ and $v$ in $\til{E}_{N}^{\geq h}$, with $D_{N, h}(u,v) = \infty$ if $u$ and $v$ are disconnected in $\til{E}_{N}^{\geq h}$.  For two subsets $A, B \subset \ctG_N$, we let
	$
	D_{N, h}(A,B) = \inf \{D_{N, h}(u,v) \,:\, u \in A, v \in B\}$. The following result is an upper bound on the chemical distance between two boundaries of a macroscopic annulus, conditioned on percolation.
	\begin{theorem}\label{2D result}
		For any fixed  $h\in \mathbb R$,  and $0 < \alpha < \beta < \gamma < 1$, let $\mbb{P}_{N,h}^{\alpha,\gamma}$ be the law of $\til{\phi}_N$ conditioned on the event $\{D_{N,h}(V_{\alpha N},\partial V_{\gamma N}) < \infty\}$. Then for any $\epsilon > 0$, there exists a constant $C$ such that
		\begin{equation}\label{2D chemical distance}
		\limsup_{N \to \infty} \mbb{P}_{N,h}^{\alpha,\gamma}(D_{N, h}(V_{\alpha N},\partial V_{\beta N}) > C N (\log N)^{\frac{1}{4}}) \leq \epsilon.
		\end{equation}
	\end{theorem}
	
	\begin{remark}\label{rem-ALP}
		Note that $\PP(D_{N,h}(V_{\alpha N},\partial V_{\gamma N}) < \infty)$ stays above 0 uniformly in $N$ (See Lemma~\ref{2D percolation probability}).  In a work in preparation by Aru--Lupu--Sep\'ulveda, it is expected that the following may be deduced as a consequence of their main results: for any $\epsilon > 0$, there exist $\delta, N_0 > 0$ such that for all $N > N_0$ and $\gamma < \beta+\delta$, we have $\PP^{\alpha,\beta}_{N,h}(D_{N,h}(V_{\alpha N},\partial V_{\gamma N}) < \infty) \geq 1- \epsilon$. Provided with this continuity of the percolation probability, one would then be able to derive from Theorem~\ref{2D result} that 
		\begin{equation}\label{eq-ALP}
		\limsup_{N \to \infty} \mbb{P}_{N,h}^{\alpha,\beta}(D_{N, h}(V_{\alpha N},\partial V_{\beta N}) > C N (\log N)^{\frac{1}{4}}) \leq \epsilon.
		\end{equation}
	\end{remark}

	For $d \geq 3$ we let $V = \ZZ^d$, let $\0$ be the origin in $\ZZ^d$, and let $\ctG$ be the metric graph associated to $V$. In the present paper, we will focus on the behavior of $p_{N, h} = \PP(\0 \stackrel{\geq h}{\longleftrightarrow} \partial V_N)$ as $N \to \infty$, where $\{\0 \stackrel{\geq h}{\longleftrightarrow} \partial V_N\}$ denotes the event that $\0$ is connected to $\partial V_N$ in $\til{E}^{\geq h} = \{v \in \ctG \,:\, \til{\phi}_{ v} \geq h\}$.  We obtain the following results for supercritical, subcritical, and critical percolation, respectively. 
	
	The first result is an explicit characterization of the probability that $\0$ is in an infinite connected component of $\til{E}^{\geq -h}$ (for the rest of this section, we take $h$ to be positive).
	\begin{theorem}\label{High d super-critical result}
		Let $\sigma^2_d = \var[\til{\phi}_\0]$. Then for any $h > 0$,
		\begin{equation}\label{Supercritical result}
		\lim_{N \to \infty} p_{N, -h} = \E\Big[ (1 - e^{-2h(\til{\phi}_\0 + h)/\sigma^2_d}) \one_{\til{\phi}_\0 > -h}\Big].
		\end{equation}
	\end{theorem}
	\noindent The next result establishes the exponential decay of $p_{N, h}$ as $N \to \infty$ for $d > 3$, with an extra log factor for $d = 3$.
	\begin{theorem}\label{High d sub-critical result}
		For any $h > 0$ and $d \geq 3$, there exists a constant $c$ such that for $N > 1$,
		\begin{align}
		p_{N, h} &\leq  \exp(-ch^2 N/\log N), \quad d =3, \label{Subcritical 3D result}\\
		p_{N, h} &\leq \exp(-ch^2 N), \quad d>3. \label{Subcritical high d result}
		\end{align}
	\end{theorem}
	\noindent The third result establishes the polynomial decay of $p_{0, N}$ as $N \to \infty$ and provides bounds on the exponent.
	\begin{theorem} \label{High d critical result}
		For any $d \geq 3$, there exist constants $C,c > 0$ such that for $N > 1$,
		\begin{align}
		\frac{c}{\sqrt{N}} \leq p_{N,0} \leq C\sqrt{\frac{\log N}{N}}, \quad d = 3, \label{Critical 3D result}\\
		\frac{c}{N^{d/2 - 1}} \leq p_{N,0} \leq \frac{C}{\sqrt{N}}, \quad d > 3. \label{Critical high d result}
		\end{align}
	\end{theorem}
	
	Finally, for $d=3$ we provide some bounds on the critical window for $h$. Below we take $h_N > 0$ to be a sequence of levels that converges to 0, and write $p_N^+$ for $p_{N,h_N}$ and $p_N^-$ for $p_{N,-h_N}$ to simplify notation.
	\begin{theorem}\label{3D critical window result}
		Suppose that there exists a constant $C > 0$ such that $h_N \leq C N^{-1/2}$ for $N \geq 1$. Then there exists a constant $c > 0$ such that for $N \geq 1$,
		\begin{equation}\label{eq-p+-lower}
		p_N^+ \geq \frac{c}{\sqrt{N}}.
		\end{equation}
		Conversely, if
		$
		\liminf_{N \to \infty} h_N \sqrt{\frac{N}{\log N \log \log N}} \geq C,
		$
		for a large enough constant $C > 0$, we have
		\begin{equation}\label{eq-p+-upper}
		\lim_{N \to \infty} \frac{p_N^+}{p_{N,0}} = 0.
		\end{equation}
		Furthermore, if there exists a constant $C > 0$ such that $h_N \leq C (\log N/ N)^{1/2}$ for $N \geq 1$, then there exists a constant $c > 0$ such that for $N > 1$
		\begin{equation}\label{eq-p--upper}
		p_N^- \leq c\sqrt{\frac{\log N}{N}}.
		\end{equation}
		Conversely, if
		$
		\lim_{N \to \infty} h_N \sqrt{\frac{N}{\log N}}= \infty,
		$ we have
		\begin{equation}\label{eq-p--lower}
		\lim_{N \to \infty} \frac{p_{N}^-}{p_{N,0}} = \infty.
		\end{equation}
	\end{theorem}

	\subsection{Related work}
	
	The chemical distance on level sets for $d=2$ has been previously studied in \cite{DingLi18} (we refer the reader to \cite{DingLi18} for another extensive discussion of related work), where it was proved that with positive probability the chemical distance between two boundaries of a macroscopic annulus is at most $Ne^{(\log N)^\alpha}$ for any fixed $\alpha>1/2$. Our Theorem~\ref{2D result} improves on \cite{DingLi18} in the following two ways:
	\begin{itemize}
		\item Instead of proving a positive probability bound as in \cite{DingLi18}, Theorem~\ref{2D result} states that the upper bound on the chemical distance holds with high probability given connectivity. At the moment, we can only show the with high probability result as in \eqref{2D chemical distance}; as noted in Remark~\ref{rem-ALP}, it is possible that with some expected future input, one would be able to derive a stronger version as in \eqref{eq-ALP}.
		
		\item The upper bound is sharpened from $Ne^{(\log N)^\alpha}$ to $N (\log N)^{1/4}$, which is somewhat surprising. In fact, the authors of the present article as well as a few people we talked to believed that the chemical distance should be at least $N \log N$.  
	\end{itemize}
	A major difference between our proof of Theorem~\ref{2D result} and the proof of the corresponding result in \cite{DingLi18} is that our proof does not rely on Makarov's theorem (on the dimension of the support of planar harmonic measures) which was a fundamental ingredient in \cite{DingLi18}. Instead of applying Makarov's theorem, we study the \emph{intrinsic structure} of the ``exploration martingale'' introduced in Section~\ref{sec:martingale}.
	
	Additionally, we remark that the result proved in \cite{DingLi18} applies to level-set percolation for Gaussian free fields on the integer lattice (as well as on the metric graph). Since percolation on the metric graph is dominated by the percolation on the integer lattice, our Theorem~\ref{2D result} implies that with non-vanishing probability the chemical distance in the level-set cluster on the integer lattice between the boundary of the annulus is $O(N (\log N)^{1/4})$.
	We feel it is possible that the methods we employ in proving Theorem~\ref{2D result} together with some technical work might be sufficient to show that the chemical distance on the integer lattice is $O(N (\log N)^{1/4})$ with high probability given connectivity. However, we prefer not to consider this problem here to avoid further complications.  Furthermore, we note that percolation clusters for level sets on the metric graph in two dimensions are of fundamental importance since their first passage sets converge to those of (continuous) Gaussian free fields \cite{ALS17, ALS18} (thus, we believe Theorem~\ref{2D result} is of substantial interest on its own). Finally, \cite{DingLi18} also established the chemical distance for critical random walk loop soup clusters. We chose not to consider proving an analogue of Theorem~\ref{2D result} for random walk clusters in the present paper.
	
	For $d\geq 3$, level-set percolation for Gaussian free fields (on metric graphs) has been studied in \cite{Lupu16}, which in particular computed  the connectivity probability between any two points and showed that the critical threshold is at $h = 0$. Our methods allow us to improve on those results by deriving more quantitative information on the phase transition, especially when $d=3$. In this case we can compute the connectivity exponent at criticality (we remark that the real contribution of the present paper is on its upper bound, since the lower bound can be deduced easily from \cite{Lupu16}), prove an almost exponential decay at subcriticality (it is quite possible that by employing a renormalization technique we can get rid of the $\log$ factor in the subcritical regime, but we chose not to consider that in the present paper), and provide an explicit description of the percolation probability in the supercritical regime (which can be rarely achieved in percolation models). Our results seem to describe the phase transition of the percolation model for the metric graph Gaussian free field in three dimensions in rather precise detail. This is somewhat interesting since percolation models in three dimensions are in general rather difficult. 
	
	That being said, we would like to mention that level-set percolation for Gaussian free fields on integer lattices for $d\geq 3$ has already been extensively studied (see \cite{MS83, BLM87, RS13, Rod14, DR15, PR15, Sznitman15, Sznitman16, Rod17, DPR17, Sznitman18}). Contrary to the case of two dimensions, percolation is substantially different on metric graphs and on integer lattices for $d\geq 3$, (roughly speaking) for the reason that there is a phase transition in higher dimensions but in two dimensions the percolation has the same qualitative behavior for any fixed $h$ (see Theorem~\ref{2D result}).  We remark that percolation on integer lattices is considerably more challenging than on metric graphs. In fact, despite intensive research, it remains an open question what the exact critical threshold is for $d\geq 3$ on integer lattices (but it was proved in \cite{DPR17} that the critical threshold is strictly positive), as well as whether a sharp phase transition exists. It would be difficult to apply methods in the present paper to prove something on the integer lattices for $d\geq 3$, for the reason that we do not have a precise control over the ``exploration martingale'' (as introduced in Section~\ref{sec:martingale}) in the case of integer lattices.  
	
	\subsection{Discussions on future directions}
	
	Our work suggests a number of interesting directions for further research,  which we list below.
	\begin{itemize}
		\item The factor of $N(\log N)^{1/4}$ in Theorem~\ref{2D result} reinstates the (now even more intriguing) question of whether the chemical distance is linear or not. Our bound of $N (\log N)^{1/4}$ strongly suggests that this is a highly delicate problem.
		\item Our method can give some non-trivial bounds on the exponent for chemical distances for $d \geq 3$ at criticality, but it seems challenging to compute the exact exponent.
		\item The difference of a factor of $\sqrt{\log N}$ between the upper and lower bounds of \eqref{Critical 3D result} hides important information about the geometry of critical clusters for $d=3$. For example, whether the capacity of the critical cluster containing $\0$ is of order $N$, conditioned on the cluster intersecting with $\partial V_N$. It would be interesting to prove an up-to-constants bound for the connectivity probability at criticality.
		\item It would be very interesting to construct an incipient infinite cluster measure for critical percolation in three dimensions, as has been done for Bernoulli percolation in two dimensions in \cite{Kesten86} (see also \cite{BS17} for a nicely streamlined presentation).
	\end{itemize}
	
	\subsection{Notation conventions and organization}
	
	For a real vector $\mathbf x$ (in any dimension), we denote by  $|\mathbf x|$  the Euclidean norm of $\mathbf x$, by $|\mathbf x|_{\ell_1}$ its $\ell_1$-norm, and by $|\mathbf x|_\infty$ its $\ell_\infty$-norm. We will also use $|A|$ to denote the cardinality of a finite set $A$. The meaning will be clear from context. We will denote by $\bar{A}$ and $A^o$ the closure and interior of a subset $A \subset \ctG$, respectively. We use $A^c$ to denote the complement of the set (or event) $A$.
	
	Throughout, we will use $\varphi$ and $\Phi$ to denote the density function and distribution function of the standard normal distribution, and $\bar{\Phi}$ to denote its survivor function. That is, $\bar{\Phi}(x) = 1 - \Phi(x)$.
	
	To simplify certain statements we use the following notation to describe the asymptotic behavior of functions of $N$ as $N$ tends to infinity. For two positive functions $f$ and $g$, we say $f(N) = O(g(N))$ as $N \to \infty$ if there exists constants $c > 0$ and $N_0 > 0$ (possibly depending on $d$, $h$, or other parameters) such that for all $N \geq N_0$, $f(N) \leq c g(N)$. We say $f(N) = o(g(N))$  if for every constant $c > 0$ there exists $N_0 > 0$ such that for all $N \geq N_0$, $f(N) \leq c g(N)$. Similarly, we say $f(N) = \Omega(g(N))$ if $g(N) = O(f(N))$, and $f(N) = \omega(g(N))$ if $g(N) = o(f(N))$. Finally, we say $f(N) = \Theta(g(N))$ if $f(N) = O(g(N))$ and $f(N) = \Omega(g(N))$.
	
	The rest of the paper is organized as follows. In Section \ref{sec:martingale} we introduce a family of martingales which is the key to proving all  the results in the present paper. In Section \ref{High d proofs} we prove the results concerning $d \geq 3$ (as the proof is substantially simpler than that for $d=2$), and in Section \ref{2D proofs} we prove Theorem \ref{2D result} concerning $d=2$ (we remark that the proof of Theorem~\ref{2D result} encapsulates all the technical ideas of the present article).

	\section{Exploration Martingale}\label{sec:martingale}
	In this section, we introduce the ``exploration martingale'' and demonstrate some of its basic properties. We note that the approach of applying martingales in the study of percolation for Gaussian free fields has appeared before (c.f. \cite{LW16, ALS17}).

	The discussion in this section applies to all dimensions, and we will denote by $\{\til{\phi}_v: v\in \ctG\}$ the Gaussian free field under consideration, without further specifying $\ctG$.
	For a finite subset $A \subset V$ (as usual, $\ctG$ is the metric graph associated to the vertex set $V$), we define the ``observable'' $X_A$ to be the average of $\til{\phi}$ on $A$
	\[
	X_A = \frac{1}{|A|} \sum_{v \in A} \til{\phi}_v.
	\]
	Let $\cI_0$ be a deterministic, closed, bounded, connected subset of $\ctG$ and let $\cI_t = \{v \,:\, D_{h}(\cI_0,v) \leq t\}$ be the closed ball of radius $t > 0$ around $\cI_0$ with respect to $D_h$, the graph distance on $\til{E}^{\geq h}$ (here we use the following convention: if $u$ and $v$ are distinct and $u \notin \til{E}^{\geq h}$ we let $D_h(u,v) = D_h(v,u) = \infty$, but for any $u \in \ctG$ we set $D_h(u,u) = 0$ even if $u \notin \til{E}^{\geq h}$). For $U \subset \ctG$, let $\mc{F}_U$ be the $\sigma$-field generated by $\{\til{\phi}_v \,:\, v \in U\}$, and $$\mc{F}_{\cI_t} = \{\mathcal E \in \mc{F}_{\tilde {\mathcal G}}: \mathcal E \cap \{\mathcal I_t \subset U\} \in \mathcal F_U \mbox{ for all open }  U\supset \mathcal I_0\} \,.$$ We then define the continuous-time martingale $M_A$ by
	\begin{equation}\label{eq-def-M}
	M_{A,t} = \E[X_A \mid \mc{F}_{\cI_t}]\,,
	\end{equation}
	We will call $M_A$ the exploration martingale with source $\cI_0$ and target $A$. Before proceeding further, we show the following measurability property of $\cI_t$.
	\begin{proposition}\label{Measurability proposition}
		For any open subset $U$ of $\til{\mc{G}}$ containing $\cI_0$ and any $t\geq 0$, we have $\{\cI_t \subset U\} \in \mc{F}_U$.
	\end{proposition}
	\begin{proof}
		Since $\cI_0$ is deterministic, we assume without loss of generality that $t > 0$. Write $U_h = \bar U \cap \til{E}^{\geq h}$, let $D_{U_{h}}$ be the graph distance on $U_h$, and $\cI_{U,t} = \{u \in \bar{U} \,:\, D_{U_h}(\cI_0, v) \leq t\}$ be the closed ball of radius $t$ around $\cI_0$ with respect to $D_{U_{h}}$. We will show that $\{\cI_t \subset U\} = \{ \cI_{U,t} \subset U\} \in \mc{F}_U$. First, it is clear that for any $u,v \in \ctG$, $D_h(u,v) \leq D_{U_h}(u,v)$ so we have $\cI_{U,t} \subseteq \cI_t$ and $\{\cI_t \subset U\} \subseteq \{\cI_{U,t}\subset U\}$. Now, assume $\cI_{U,t} \subset U$ and that there exists $v \in \cI_t$ with $v \notin U$. Then there exists $0< s \leq t$ and a path $\Gamma : [0,s] \to \ctG$ parametrized by $D_{h}$ with $\Gamma(0) = u_0 \in \partial \cI_0$, $\Gamma(s) = v$, and $\Gamma([0,s]) \subseteq \til{E}^{\geq h}$. Setting $s' = \sup\{x \in [0,s] \,:\, \Gamma(x) \in \cI_{U,t}\}$ we have $\Gamma(s') \in \cI_{U,t} \subset U$ and thus $s' < s$. Hence,  there exists an open neighborhood $(a,b) $ of $s'$ such that $\Gamma((a,b)) \subset U$. By assumption, $\Gamma((a,b)) \subset \til{E}^{\geq h}$, so we get $\Gamma((a,b)) \subset U_h$ and thus $\Gamma((a,b)) \subset \cI_{U,t}$, which contradicts the maximality of $s'$. This concludes the proof that $\{\cI_t \subset U\} = \{\cI_{U,t} \subset U\}$ and thus the proposition follows.
	\end{proof}
	Proposition \ref{Measurability proposition}, together with the following strong Markov property of the Gaussian free field (see \cite[Theorem 4 in Chapter 2, Section 2.4] {Rozanov82} for a proof) will provide useful formulas for $M_{A}$ and its quadratic variation $\langle M_A \rangle$.
	\begin{theorem}\label{Gibbs-Markov}
		Let $\cK$ be a random compact connected subset of $\ctG$ such that for every deterministic open subset $U$ of $\ctG$, the event $\{\cK \subseteq U\} \in \mc{F}_U$. Then conditioned on $\cK$ and $\mc{F}_{\cK}$, $\{\til{\phi}_v \,:\, v \in \ctG \setminus \cK\}$ is equal in distribution to $\{\E[\til{\phi}_{v} \mid \mc{F}_\cK] + \til{\psi}_{v}\,:\, v \in \ctG\setminus \cK\}$ where $\til{\psi}$ is a Gaussian free field (with Dirichlet boundary condition) on $\ctG \setminus \cK$. Additionally, for $v \in \ctG \setminus \cK$ and $T = \inf \{ t \geq 0 \,:\, \tB_t \in \cK \}$
		\begin{equation}\label{GFF conditional mean}
		\E[\til{\phi}_v \mid \mc{F}_{\cK}] = \E_{v} [ \til{\phi}_{\tB_T} \one_{T < \zeta}\mid \mc{F}_{\cK}] = \sum_{u \in \partial \cK} \Hm(v,u; \cK) \til{\phi}_u.
		\end{equation}
		Here the harmonic measure $\Hm$ is given by $\Hm(v,u;\cK) = \PP_v(T < \zeta,\, \tB_T = u)$.
	\end{theorem}
	Note that if  $v \in \cK$, the harmonic measure is a point mass of mass 1 at $v$. To avoid having to account for this case separately, we will sometimes let the sum in \eqref{GFF conditional mean} range over all $u \in \cK$. The summation notation is justified since in either case the harmonic measure is supported on a finite number of points. To simplify notation, we will write $\Hm_t(v,u)$ for $\Hm(v,u;\cI_t)$ and $\Hm_t(A,u)$ for $|A|^{-1} \sum_{v \in A} \Hm_t(v,u)$. With this notation, \eqref{GFF conditional mean} gives
	\begin{equation}\label{Martingale value}
	M_{A,t} = \sum_{u \in \cI_t} \Hm_t(A,u) \til{\phi}_u.
	\end{equation}
	Sometimes, it will be useful to consider the total mass of the harmonic measure. For a set $B \subset \til{\mc{G}}$, we let $\Hm(v,B) = \sum_{w \in B} \Hm(v,w;B)$, and $\Hm(A,B) = |A|^{-1}\sum_{v \in A} \Hm(v,B)$. To simplify notation, we will write $\pi_{A,t}$ for $\Hm(A, \mc I_t)$. For a set $U$, we let $\tau_U = \inf\{t \geq 0 \,:\, \til B_t \in U\}$ be the hitting time of $U$ by $\til B$.
	
	The following lemma establishes the continuity of $M_A$ and will allow us to compute its quadratic variation. We defer the proof to the end of this section.
	\begin{lemma}\label{Continuity lemma}
		$M_{A,t}$ and $\var[X_A \mid \mc{F}_{\cI_t}]$ are almost-surely continuous as functions of $t$.
	\end{lemma}
	Recall that for a continuous martingale $M$ its quadratic variation $\langle M \rangle$ is the unique increasing continuous process vanishing at zero such that $M^2 - \langle M \rangle$ is a martingale (see \cite[Theorem 1.3 in Chapter IV]{RevusYor99}). From this we deduce the following.
	\begin{corollary}\label{QV form}
		The quadratic variation of $M_A$ is given by
		\[
		\langle M_A \rangle_t = \var[X_A \mid \mc{F}_{\cI_0}] - \var[X_A \mid \mc{F}_{\cI_t}].
		\]
	\end{corollary}
	\begin{proof}
		By definition the process $\var[X_A \mid \mc{F}_{\cI_0}] - \var[X_A \mid \mc{F}_{\cI_t}]$ is adapted to $\mc{F}_{\cI_t}$ and vanishes at zero. It is increasing because $\cI_t$ is increasing (so $\var[X_A \mid \mc{F}_{\cI_t}]$ is decreasing), and it is continuous by Lemma~\ref{Continuity lemma}. Therefore by the characterization of $\langle M_A \rangle$ we only need to show that the process $\{Y_t\,:\, t \geq 0\}$ is a martingale, where
		\[
		Y_t = M_{A, t}^2 - \var[X_A \mid \mc{F}_{\cI_0}] + \var[X_A \mid \mc{F}_{\cI_t}].
		\]
		To this end, note that for any times $0 \leq s < t$,
		\begin{align*}
		\E[M_{A,t}^2 \mid \mc{F}_{\cI_s}] &= M_{A,s}^2 + \E[(M_{A,t} - M_{A,s})^2 \mid \mc{F}_{\cI_s}]\,, \\
		\E[\var[X_A \mid \mc{F}_{\cI_t}] \mid \mc{F}_{\cI_s}] &= \E[(X_A - M_{A,t})^2 \mid \mc{F}_{\cI_s}]\\
		&= \var[X_A \mid \mc{F}_{\cI_s}] - \E[(M_{A,t} - M_{A,s})^2 \mid \mc{F}_{\cI_s}]. 
		\end{align*}
		We can then calculate $\E[ Y_t \mid \mc{F}_{\cI_s}]$ as follows
		\begin{align*}
		\E[Y_t \mid \mc{F}_{\cI_s}] &= \E[M_{A,t}^2 \mid \mc{F}_{\cI_s}] + \E[\var[X_A \mid \mc{F}_{\cI_t}] \mid \mc{F}_{\cI_s}] -  \var[X_A \mid \mc{F}_{\cI_0}] \\
		&= M_{A,s}^2 + \var[X_A \mid \mc{F}_{\cI_s}] -\var[X_A \mid \mc{F}_{\cI_0}] \\
		&= Y_s\,,
		\end{align*}
		completing the verification that $\{Y_t\,:\, t \geq 0\}$ is a martingale.
	\end{proof}
Now, we present some formulas which will be used below to compute quadratic variations. Let $G_t$ be the Green's function on $\ctG \setminus \cI_t$. We get from Theorem~\ref{Gibbs-Markov} that
	\begin{align}
	\var[X_A \mid \mc{F}_{\cI_t}] &= \frac{1}{|A|^2} \sum_{v, v' \in A} G_t(v,v'). \label{conditional variance}
	\end{align}
	We also note that $G_t$ may be written in terms of $G$ and $\Hm_t$ as follows
	\begin{equation}\label{eq-G-t}
	G_t(v,v') = G(v,v') - \sum_{w \in \mc{I}_t}\Hm_t(v,w)G(w,v').
	\end{equation}
	Additionally, for $0 \leq s \leq t$ and $v,v' \in A$ we have the following two expressions for $G_s(v,v') - G_t(v,v')$.
	\begin{align*}
	G_s(v,v') - G_t(v,v') &= \sum_{w \in \mc I_t} \Hm_t(v,w)G_s(w,v') \\
	&= \sum_{w \in \mc I_t} \Hm_t(v,w)G(w,v') - \sum_{w' \in \mc I_s} \Hm_s(v,w')G(w',v').
	\end{align*}
	
	Next, we recall that the quadratic variation relates $M$ to Brownian motion. In particular \cite[Theorem 1.7 in Chapter V]{RevusYor99}, stated below, gives the appropriate extension of the Dubins-Schwarz theorem for martingales of bounded quadratic variation.
	\begin{theorem}\label{Time-change theorem}
		Let be $M$ a continuous martingale, $T_t = \inf\{s \,:\, \langle M \rangle_s > t\}$, and $W$ be the following process
		\[
		W_t = \begin{cases}
		M_{T_t} - M_0 & t < \langle M \rangle_\infty\, ,\\
		M_{\infty} - M_0 & t \geq \langle M \rangle_\infty.
		\end{cases}
		\]
		Then $W$ is a Brownian motion stopped at $\langle M \rangle_\infty$.
	\end{theorem}
	When applying this theorem, we will generally denote by $B$ a Brownian motion which satisfies $B_t = M_{T_t} - M_0$ for $t < \langle M \rangle_\infty$ but is not stopped at $\langle M \rangle_\infty$, so that $W_t = B_{t \wedge \langle M \rangle_\infty}$. We note that by Theorem~\ref{Gibbs-Markov}, $\{M_{A,t}-M_{A,0} \,:\, t \geq 0\}$ is independent of $\mc{F}_{\mc I_0}$, so we will generally take $B$ independent of $\mc{F}_{\mc I_0}$ as well.
	Finally, we prove that $M_{A,t}$ and $\var[X_A \mid \mc{F}_{\cI_t}]$ are indeed continuous.
	\begin{proof}[Proof of Lemma \ref{Continuity lemma}]
		It suffices to show that for any $v,v' \in V$, $\E[\til{\phi}_v \mid \mc{F}_{\cI_t}]$ and $\cov[\til{\phi}_v \til{\phi}_v' \mid \mc{F}_{\cI_t}]$ are continuous (then it is clear that $M_{A,t}$ and $\var[X_A \mid \mc{F}_{\cI_t}]$ are averages of a finite number of continuous functions and are thus continuous). Since both functions are constant for $t \geq D_h(\cI_0, v)$, we let $0 \leq t \leq D_h(\cI_0, v)$. By \eqref{GFF conditional mean} and \eqref{eq-G-t} it suffices to show that for any continuous function $f$ on $\ctG$, the following function is continuous
		\[
		F(t) = \sum_{u \in \partial \cI_t} \Hm_t(v,u) f(u).
		\]
		Let $D_{\ell_1}$ denote the graph distance on $\ctG$ (i.e., $\ell_1$ distance on $\mbb R^d$). Since $K = |\partial \cI_t| < \infty$,
		\[
		\delta_1 = \min\{D_{\ell_1}(u, (\partial \cI_t \cup V) \setminus \{u\}) \,:\, u \in \partial \cI_t\} > 0.
		\]
		For $s$ such that $|t-s| < \delta_1/2$ and $u \in \partial \cI_t$, let $\psi_s(u) = \partial \cI_s \cap B_{\ell_1}(u,\delta_1/2)$ (here $B_{\ell_1}(u,\delta_1/2)$ is the open ball of radius $\delta_1/2$ around $u$ with respect to $D_{\ell_1}$). The sets $\psi_s(u)$ are non-empty and disjoint. If $s >t$, $\partial \cI_s = \cup_{u \in \partial \cI_t} \psi_s(u)$; if $s < t$, we let $\mc{R}_s = \partial \cI_s \setminus (\cup_{u \in \partial \cI_t} \psi_s(u))$. We have
		\begin{align*}
		|F(t) - F(s)| \leq &\sum_{u \in \partial \cI_t} \Big|\Hm_t(v,u) f(u) - \sum_{u' \in \psi_s(u)} \Hm_s(v,u') f(u')\Big| \\
		&+ \one_{s < t} \Big|\sum_{u' \in \mc{R}_s} \Hm_s(v,u') f(u')\Big|.
		\end{align*}
		Since $\cI_t$ is compact, $M = \max \{ f(u) \,:\, u \in \cI_t\} < \infty$. Since $\partial \cI_t$ is finite, for any $\epsilon > 0$ there exists $0 < \delta_2 \leq \delta_1/2$ such that $|f(u) - f(u')| < \epsilon/2$ for $u \in \partial \cI_t$ and $u' \in B_{\ell_1}(u,\delta_2)$. Thus, for $|t-s| < \delta_2$
		\[
		|F(t) - F(s)| \leq M \sum_{u \in \partial \cI_t} \Big|\Hm_t(v,u) - \Hm_s(v, \psi_s(u))\Big| + \one_{s < t} M \Hm_s(v, \mc{R}_s) + \frac{\epsilon}{2}.
		\]
		Finally, it follows from the construction of $\til{B}_t$ (by considering the excursions of a standard Brownian motion) that for any $u \in \ctG$, $b \leq D_{\ell_1}(u,V\setminus\{u\})$, and $u' \in \ctG$ such that $a = |u-u'|_{\ell_1} \leq b$,
		\[
		\Hm(u,u' \,;\, \{u'\} \cup \partial B_{\ell_1}(u,b)) \geq \frac{b}{b + (2d-1)a}.
		\]
		Combining this with the previous bound it follows from a straightforward calculation that there exists $\delta_3 \leq \delta_2$ such that if $|t-s| < \delta_3$,
		\begin{equation*}
		|F(t) - F(s)| \leq \epsilon. \qedhere
		\end{equation*}
	\end{proof}

	\section{Percolation in three and higher dimensions}\label{High d proofs}
	In the case $d \geq 3$, we let the vertex set $V = \ZZ^d$ be the whole lattice and study the behavior of $p_{N, h} = \PP(\0 \stackrel{\geq h}{\longleftrightarrow} \partial V_N)$ as $N \to \infty$. The theorems characterizing the behavior of $p_{N,h}$ will follow from Proposition~\ref{Law of capacity}, stated below. The idea is to consider an exploration martingale with source set $\cI_0 = \{\0\}$ and target set $A = \{x_K\}$, with $|x_K| = K$ as $K \to \infty$. We note that the calculations below are valid for any sequence $\{x_K\}_{K \geq 1}$ satisfying this condition so we do not specify $x_K$ further. For ease of notation we will write $M_{K}$ for the exploration martingale instead of $M_{x_K}$. Recall $\sigma_d^2 = G(\0,\0)$ and note that by translation invariance $\sigma_d^2 = G(u,u)$ for all $u \in \ZZ^d$. Recall also the process $\pi_{K,t} = \sum_{v \in \mc{I}_t} \Hm_t(x_K,v)$.
	\begin{proposition}\label{Law of capacity}
		Let $B$ be a standard one-dimensional Brownian motion that is independent of $\til{\phi}_0$ and $\tau_h$ be the following stopping time
		\begin{align*}
		\tau_{h} &= \inf \left\{ t \geq 0 \,:\, B_t \leq h t - \frac{\til{\phi}_{\0} - h}{\sigma_d^2} \right\}.
		\end{align*}
		We have
		\[
		\lim_{K \to \infty} \mbb{P}\left(\frac{\pi_{K,\infty} - \pi_{K,0}}{G(\0,x_K)}\leq s\right) = \mbb{P}(\tau_h \leq s), \quad \forall s \geq 0.
		\]
		That is, $(\pi_{K,\infty}-\pi_{K,0})/G(\0,x_K)$ converges in law to $\tau_h$ as $K \to \infty$.
	\end{proposition}
The proposition will follow from the following two lemmas, whose proofs are deferred to the end of the section. The first lemma will allow us to relate the quadratic variation of $M_K$ to the harmonic measure of $\mc{I}_t$.
	\begin{lemma}\label{QV convergence}
		Let $N \geq 0$ be given and $\mc I$ be any compact connected subset of $\til{\mc{G}}$ satisfying $\0 \in \mc I \subseteq V_N$. Let $\Hm(x_K,\mathcal I)$ be the probability that a metric graph Brownian motion started at $x_K$ hits $\mathcal I$. We have
		\[
		\frac{\var[\til\phi_{x_K}\mid \mathcal F_{\0}]
			- \var[\til\phi_{x_K}\mid \mathcal F_{\mathcal I}]}{G(\0,x_K)(
			\Hm(x_K,\mathcal I)
			- \Hm(x_K,\0))}
			= 1 + O\left(\frac{N}{K}\right) \text{ as } K \to \infty.
		\]
	\end{lemma}
The second lemma gives upper and lower bounds for $(\pi_{K,t}-\pi_{K,0})/G(\0,x_K)$ at the time $\mc{I}_t$ hits $\partial V_N$.
	\begin{lemma}\label{Capacity bounds}
		Let $f_1$ and $f_2$ be the following functions
		\begin{align*}
		f_1(N) &= \limsup_{K \to \infty} \frac{\Hm(x_K,V_N) - \Hm(x_K,\0)}{G(\0,x_K)},\\
		f_2(N) &= \liminf_{K \to \infty}\inf_{\mc I} \frac{\Hm(x_K, \mc{I}) - \Hm(x_K,\0)}{G(\0,x_K)},
		\end{align*}
		where the infimum in the definition of $f_2$ is taken over all compact, connected sets containing $\0$ and intersecting $\partial V_N$. We have
		\begin{align*}
		f_1(N) &= O(N^{d-2}),\\
		f_2(N) &= \Omega\left(\frac{N}{\log(N)^{\one_{d = 3}}}\right).
		\end{align*}
	\end{lemma}
As promised, taking these two lemmas as given, Proposition~\ref{Law of capacity} follows easily.
	\begin{proof}[Proof of Proposition \ref{Law of capacity}]
		 First, we note that the case $s = 0$ is trivial, as $\{\pi_{K,\infty} = \pi_{K,0}\} = \{\til\phi_\0 \leq h\} = \{\tau_h = 0\}$. Therefore, we will assume that $\til\phi_\0 > 0$ and take $s > 0$. We let $f_2$ be as in Lemma~\ref{Capacity bounds}, and $N$ be such that $f_2(N) > 4s$. By Lemma~\ref{QV convergence}, for any $0 < \epsilon < 1/2$ there exists $K_0$ such that for all $K \geq K_0$ the following holds almost surely. For all $0 \leq t \leq D_h(\0,\partial V_N)$,
		 \begin{equation}\label{QV approximate equality}
		 \frac{\pi_{K,t} - \pi_{K,0}}{G(\0,x_K)}(1 - \epsilon) \leq \frac{\langle M \rangle_t}{G(\0,x_K)^2} \leq \frac{\pi_{K,t} - \pi_{K,0}}{G(\0,x_K)}(1 + \epsilon).
		 \end{equation}
		 Additionally, by our choice of $N$ we can take $K_0$ large enough that for any $K \geq K_0$ and any compact connected set $\mc I$ connecting $\0$ to $\partial V_N$ we have
		 \[
		 \frac{\Hm(x_K,\mc I) - \Hm(x_K,\0)}{G(\0,x_K)} > 4s.
		 \]
		 
		 By the definition of $M_K$, we have $M_K \geq h \pi_{K,t}$, with equality if and only if $\pi_{K,t} = \pi_{K,\infty}$. Letting $\eta_h = \inf\{t \geq 0\,:\, M_{K,t} = h \pi_{K,t}\} \wedge D_h(\0, \partial V_N)$ we have by our assumptions on $K$
		 \[
		 \left\{\frac{\pi_{K,\infty}-\pi_{K,0}}{G(\0,x_K)} \leq s\right\} = \left\{\frac{\pi_{K,\eta_h}-\pi_{K,0}}{G(\0,x_K)} \leq s\right\}
		 \]
		Now, we let $T_t = \inf\{s \geq 0 \,:\, \langle M_K \rangle_s/G(\0,x_K)^2 > t\}$ and $B$ be a standard Brownian motion satisfying $B_t = (M_{K,T_t} - M_{K,0})/G(\0,x_K)$ for $0 \leq t \leq \langle M_K\rangle_\infty/G(\0,x_K)^2$. As usual, we take $B_t$ independent of $\til\phi_\0$. We note
		\[
		\eta_h = \inf\left\{t \geq 0\,:\, \frac{M_{K,t}-M_{K,0}}{G(\0,x_K)} \leq h \frac{\pi_{K,t} - \pi_{K,0}}{G(\0,x_K)} - \frac{\til\phi_\0 - h}{\sigma^2_d}\right\},
		\]
		where we used the fact that $G(\0,x_K) = \pi_{K,0}\sigma^2_d$. Therefore, letting $\tau'_h = \langle M_K \rangle_{\eta_h}/G(\0,x_K)^2$ and $\tau_s = \inf\{t \geq 0 \,:\, B_t \leq st - (\til\phi_\0 - h)/\sigma^2_d\}$ we have by \eqref{QV approximate equality}
		\begin{align*}
		\tau_{h/(1+\epsilon)} \wedge (2s) \leq \tau'_h \leq \tau_{h/(1-\epsilon)}, \quad h < 0, \\
		\tau_{h/(1-\epsilon)} \wedge (2s) \leq \tau'_h \leq \tau_{h/(1+\epsilon)}, \quad h \geq 0.
		\end{align*}
		By another application of \eqref{QV approximate equality}, we have
		\[
		\frac{\tau_h'}{1+\epsilon} \leq \frac{\pi_{K,\eta_h}-\pi_{K,0}}{G(\0,x_K)} \leq \frac{\tau_h'}{1-\epsilon}.
		\]
		Therefore, we conclude that for all $K \geq K_0$,
		\begin{align*}
		\mbb{P}\left(\frac{\tau_{h/(1-\epsilon)}}{1-\epsilon} \leq s \right) \leq \mbb{P}\left( \frac{\pi_{K,\eta_h}-\pi_{K,0}}{G(\0,x_K)} \leq s\right) \leq \mbb{P}\left(\frac{\tau_{h/(1+\epsilon)}}{1+\epsilon} \leq s \right), \quad h < 0, \\
		\mbb{P}\left(\frac{\tau_{h/(1+\epsilon)}}{1-\epsilon} \leq s \right) \leq \mbb{P}\left( \frac{\pi_{K,\eta_h}-\pi_{K,0}}{G(\0,x_K)} \leq s\right) \leq \mbb{P}\left(\frac{\tau_{h/(1-\epsilon)}}{1+\epsilon} \leq s \right), \quad h \geq 0.
		\end{align*}
		Letting $\epsilon \downarrow 0$ and noting that $\lim_{\epsilon \downarrow 0} \tau_{h/(1+\epsilon)} = \lim_{\epsilon \downarrow 0} \tau_{h/(1-\epsilon)} =\tau_h$ almost surely (see, e.g. Proposition~\ref{Hitting times with drift} below) then concludes the proof.
	\end{proof}
	
	The following proposition (c.f. \cite[Equation 2.0.2 in Part II]{BrownianMotionHandbook}) gives the distribution of $\tau_h$, and will allow us to obtain quantitative estimates from Proposition~\ref{Law of capacity}.
	\begin{proposition}\label{Hitting times with drift}
		For $m \in \mathbb{R}$ and $b > 0$, let $\tau = \inf \{ t > 0 \,:\, B_t \leq mt - b\}$. Then for $T > 0$,
		\[
		\PP(\tau \leq T) = \bar{\Phi}\left(\frac{b}{\sqrt{T}} - m\sqrt{T} \right) + e^{2bm}\bar{\Phi}\left(\frac{b}{\sqrt{T}} + m \sqrt{T}\right).
		\]
	\end{proposition}
	We note that it follows from this proposition that $\mbb{P}(\tau_h \leq s)$ is continuous in $s$, so we will not distinguish between $\{\tau_h \leq s\}$ and $\{\tau_h < s\}$ during the rest of this section.
	
	\subsection{Proof of main theorems}
	In this subsection we prove Theorems \ref{High d super-critical result}, \ref{High d sub-critical result}, and \ref{High d critical result}. The following corollary of Proposition~\ref{Law of capacity} will be used in each case
	\begin{corollary}\label{Capacity and radius}
	Let $f_1$ and $f_2$ be as in Lemma~\ref{Capacity bounds}, and $\tau_h$ be as in Proposition~\ref{Law of capacity}. We have
	\[
	\mbb{P}(\tau_h \geq f_1(N)) \leq p_{N,h} \leq \mbb{P}(\tau_h \geq f_2(N))
	\]
	\end{corollary}
	\begin{proof}
		Let $f_{1,K}$ and $f_{2,K}$ be defined as follows
		\[
		f_{1,K} = \frac{\Hm(x_K,V_N)-\Hm(x_K,\0)}{G(\0,x_K)}, \quad f_{2,K} = \inf_{\mc I} \frac{\Hm(x_K,\mc  I)-\Hm(x_K,\0)}{G(\0,x_K)},
		\]
		where as in Lemma~\ref{Capacity bounds}, the infimum is over all compact connected subsets of $\til{\mc G}$ connecting $\0$ to $\partial V_N$. We have for all $K$,
		\[
		\mbb{P}\left(\frac{\pi_{K,\infty} - \pi_{K,0}}{G(\0,x_K)} \geq f_{1,K}(N)\right) \leq p_{h,N}\leq \mbb{P}\left(\frac{\pi_{K,\infty} - \pi_{K,0}}{G(\0,x_K)} \geq f_{2,K}(N)\right).
		\]
		Taking limits as $K \to \infty$ we get for all $\epsilon > 0$
		\[
		\mbb{P}(\tau_h \geq f_1(N) + \epsilon) \leq p_{N,h} \leq \mbb{P}(\tau_h \geq f_2(N) - \epsilon).
		\]
		Since the distribution of $\tau_h$ is continuous, the conclusion follows.
	\end{proof}
	\begin{proof}[Proof of Theorem \ref{High d super-critical result}]
	The theorem is now a direct consequence of Corollary~\ref{Capacity and radius}, the fact that $f_2(N) \to \infty$ as $N \to \infty$, and the following corollary to Proposition~\ref{Hitting times with drift}. For $h > 0$ and $b > 0$, the probability that $B_t > -ht - b$ for all $t$ is $1 - \exp(-2bh)$. Replacing $b = \one_{\til \phi_\0 > -h}(\til\phi_\0 + h)/\sigma^2_d$ then gives the desired result.
	\end{proof}
	\begin{proof}[Proof of Theorem \ref{High d sub-critical result}]
	In this case we use the following simple bound, which is a direct consequence of Proposition~\ref{Hitting times with drift} and the (easily checked) fact that $\bar\Phi(x)\leq \exp(-x^2/2)$ for all $x \geq 0$,
		\begin{align*}
		\mbb{P}(\tau_h \geq s \mid \mathcal F_{\0}) \leq \bar\Phi\left(h\sqrt{s} - \frac{\til\phi_\0-h}{\sigma^2_d\sqrt{s}}\right) \leq \exp\left(-\frac{h^2s}{2} + \frac{h(\til\phi_\0 - h)}{\sigma^2_d}\right).
		\end{align*}
		This gives
		\begin{align*}
		p_{N, h} \leq \mbb{P}(\tau_h \geq f_2(N)) &\leq \exp\left(-\frac{h^2 f_2(N)}{2}\right) \E\left[\exp\left(\frac{h(\til\phi_\0 - h)}{\sigma^2_d}\right)\right] \\
		&= \exp\left[-\Omega\left(\frac{h^2 N}{(\log N)^{\one_{d = 3}}}\right)\right]. \qedhere
		\end{align*}
	\end{proof}
	\begin{proof}[Proof of Theorem \ref{High d critical result}]
		As in the previous cases, we apply Proposition \ref{Hitting times with drift}, Corollary~\ref{Capacity and radius}, and Lemma~\ref{Capacity bounds} to obtain
		\begin{align*}
		p_{N,0} &\geq \mbb{P}(\tau_0 \geq f_1(N))\\
		&= \mbb{E}\left[\left(\Phi\left(\frac{\til\phi_\0}{\sigma^2_d\sqrt{f_1(N)}}\right) - \Phi\left(-\frac{\til\phi_\0}{\sigma^2_d\sqrt{f_1(N)}} \right)\right) \one_{\til{\phi}_\0 > 0}\right]\\
		&= \Omega\left( \frac{1}{N^{d/2-1}}\right).
		\end{align*}
		The upper bound follows by the same reasoning
		\begin{align*}
		p_{N,0} &\leq \mbb{P}(\tau_0 \geq f_2(N))\\
		&= \mbb{E}\left[\left(\Phi\left(\frac{\til\phi_\0}{\sigma^2_d\sqrt{f_2(N)}}\right) - \Phi\left(-\frac{\til\phi_\0}{\sigma^2_d\sqrt{f_2(N)}} \right)\right) \one_{\til{\phi}_\0 > 0}\right]\\
		&= O\left( \frac{\sqrt{\log N}^{\one_{d=3}}}{\sqrt{N}}\right).
		\end{align*}
	\end{proof}
	
	\subsection{Critical window in three dimensions}
	In this section we prove Theorem \ref{3D critical window result}. That is, we give rates of decay for $h$ (now considered as a function of $N$) such that $p_{N, \pm h}$ is of the same order as $p_{N,0}$. We will only consider the case $d = 3$ in this subsection. Throughout, we let $h_N > 0$ be a sequence such that $h_N \to 0$. To simplify notation, we will write $\sigma^2$ for $\sigma^2_3$, $p_N^\pm$ for $p_{N,\pm h_N}$, and similarly with other quantities. Additionally, we write $(a)_+$ for $\max\{a,0\}$ and will use $(\til{\phi}_\0 \mp h_N)_+$ instead of $(\til{\phi}_\0 \mp h_N)$ when applying Proposition~\ref{Law of capacity} to avoid writing $\one_{\til{\phi}_\0 > \pm h_N}$ when taking expectations.
	
	We first prove \eqref{eq-p+-lower}. Letting $b = (\til{\phi}_\0 - h_N)_+/\sigma^2$, we have from Proposition~\ref{Capacity and radius} and Proposition~\ref{Hitting times with drift}
	\begin{align*}
	p_N^+ &\geq \E\left[\bar{\Phi}\left(h_N\sqrt{f_1(N)} - \frac{b}{\sqrt{f_1(N)}}\right) - e^{2h_Nb}\bar{\Phi}\left(h_N\sqrt{f_1(N)} + \frac{b}{\sqrt{f_1(N)}}\right) \right]\numberthis \label{eq-critical-window-p-+}.
	\end{align*}
	To bound the right hand side of the preceding inequality we use the following  lemma.
	\begin{lemma}\label{Elementary hitting time bound}
		Let $x \in \mbb{R}$ and $y \geq 0$. Define $f$ and $g$ by
		\begin{align*}
		f(x,y) = \bar{\Phi}(x - y) - e^{2xy}\bar{\Phi}(x + y) \mbox{ and }
		g(x,y) = 1 - \frac{x \bar{\Phi}(x+y)}{\varphi(x+y)}.
		\end{align*}
		We have $g(x,y) > 0$ and
		\begin{equation}\label{eq-analysis-inequality}
		\begin{split}
		2g(x,0)\left[\Phi(x) - \Phi(x-y)\right]\leq f(x,y) &\leq 2g(x,y)\left[\Phi(x) - \Phi(x-y)\right] \quad x \geq 0,\\
		2g(x,y)\left[\Phi(x) - \Phi(x-y)\right]\leq f(x,y) &\leq 2g(x,0)\left[\Phi(x) - \Phi(x-y)\right] \quad x \leq 0.
		\end{split}
		\end{equation}
	\end{lemma}
	\begin{proof}
		It is clear that $g(x,y) > 0$ for $x \leq 0$. For $x > 0$, the fact that $g(x,0) > 0$ is equivalent to the well-known (and straightforward to check) bound $\bar \Phi(x) < \varphi(x)/x$ for all $x > 0$ and it directly implies $g(x,y) \geq g(x+y,0) > 0$. 
		
		To prove \eqref{eq-analysis-inequality}, note that
		\[
		\frac{\partial f}{\partial y} (x,y) = 2\varphi(x-y) - 2x e^{2xy} \bar{\Phi}(x+y) = 2g(x,y)\varphi(x-y).
		\] 
		Using the fact that $\bar{\Phi}(x)/\varphi(x)$ is decreasing in $x$ (for all values of $x$) we conclude that $g(x,y)$ is decreasing in $y$ for $x < 0$ and increasing in $y$ for $x > 0$. The desired bounds follow by integrating $\partial f/ \partial y$. For instance, for $x > 0$ we have
		\begin{align*}
		f(x,y) = \int_0^y \frac{\partial f}{\partial y}(x,s)ds &\leq 2g(x,y) \int_0^y \varphi(x-s)ds \\
		&= 2g(x,y)\left[\Phi(x) - \Phi(x-y)\right]\,.
		\end{align*}
		The other three bounds follow by similar arguments.
	\end{proof}
	Note now that $h_N = O(N^{-1/2})$ implies $h_N\sqrt{f_1(N)} = O(1)$ and therefore $g(h_N \sqrt{f_1(N)},0) = \Omega(1)$. Recalling \eqref{eq-critical-window-p-+} and applying  Lemma~\ref{Elementary hitting time bound} with $x = h_N\sqrt{f_1(N)}$ and $y = b/\sqrt{f_1(N)}$ gives
	\begin{align*}
	p_N^+ &\geq 2g(h_N\sqrt{f_1(N)},0)\left(\Phi\left(h_N\sqrt{f_1(N)}\right) - \mbb{E}\left[\Phi\left(h_N\sqrt{f_1(N)} - \frac{b}{\sqrt{f_1(N)}}\right) \right]\right)\\
	&= \Omega\left(\frac{1}{\sqrt{N}}\right)\,,
	\end{align*}
	which proves \eqref{eq-p+-lower}.
	
	We turn next to \eqref{eq-p+-upper}. We let $b$ be as above, and note the trivial bounds $g(x,y) \leq 1$ and $\Phi(x) - \Phi(x-y) \leq y \varphi((x-y)_+)$ which are valid for all $x,y \geq 0$. Combining this with Lemma~\ref{Elementary hitting time bound} we obtain
	\begin{align*}
	p_N^+ &\leq \mbb{P}(\tau_{N}^+ \geq f_2(N))
	\\&\leq \E\left[ \frac{2b}{\sqrt{f_2(N)}}\varphi\left(\left(h_N\sqrt{f_2(N)}-\frac{b}{\sqrt{f_2(N)}}\right)_+\right)\right] \\
	&\leq \frac{\exp(-h_N^2f_2(N)/2)}{\sqrt{f_2(N)}}\E\left[b e^{h_Nb} \right].
	\end{align*}
	Recalling $f_2(N) = \Omega(N/\log N)$, we have under the assumption
	$
	\liminf_{N \to \infty} \frac{h_N \sqrt{N}}{\sqrt{\log N \log \log N}} \geq C
	$
	for a large enough constant $C$, that $e^{-h_N^2f_2(N)/2} = o(\sqrt{\log N})$. This gives
	$
	p_N^+ = o\left(\frac{1}{\sqrt{N}}\right)
	$ as required for \eqref{eq-p+-upper}.
	
	The bounds on $p_N^-$ are obtained by a similar argument. Let $b = (\til{\phi}_\0 + h_N)_+/\sigma^2$, and note that for $x,y \geq 0$ we have $g(-x,0) \leq 1 + x/\varphi(x)$ and $\Phi(-x) - \Phi(-x-y) \leq y \varphi(x)$. Applying Lemma~\ref{Elementary hitting time bound} with $x = h_N\sqrt{f_2(N)}$ and $y = b/\sqrt{f_2(N)}$ this gives
	\begin{align*}
	p_N^- &\leq \mbb{P}(\tau_{N}^- \geq f_2(N)) \\
	&\leq 2[\varphi(h_N\sqrt{f_2(N)}) + h_N\sqrt{f_2(N)}] \E\left[\frac{b}{\sqrt{f_2(N)}}\right]\\
	&= O(h_N)\,.
	\end{align*}
	Recalling $f_2(N) = \Omega(N/\log(N))$, we see that if $h_N = O(\sqrt{\log N/N})$, then the above implies
	$
	p_N^- = O\left(\sqrt{\frac{\log N}{N}}\right)
	$, as required for \eqref{eq-p--upper}.
	
	Conversely, we have for $x, y \geq 0$ that $g(-x,y) \geq x\bar{\Phi}(-x+y)/\varphi(-x+y)$ and $\Phi(-x) - \Phi(-x-y) \geq y \varphi(x+y)$, so Lemma~\ref{Elementary hitting time bound} gives
	\[
	\bar{\Phi}(-x - y) - e^{-2xy}\bar{\Phi}(-x + y) \geq 2xy e^{-2xy} \bar\Phi(-x+y).
	\]
	Letting $x = h_N\sqrt{f_1(N)}$ and $y = b/\sqrt{f_1(N)}$ we obtain
	\begin{align*}
	p_N^- &\geq \mbb{P}(\tau_N^- \geq f_1(N))\\
	&\geq \E\left[2h_Nb e^{-2h_Nb} \bar{\Phi}\left(-h_N\sqrt{f_1(N)} + \frac{b}{\sqrt{f_1(N)}}\right)\right]\,.
	\end{align*}
	It follows that if $h_N = o(1)$ and $h_N = \Omega(N^{-1/2})$, $p_N^- = \Omega(h_N)$. In particular, $h_N = \omega(\sqrt{\log N/N})$ implies $p_N^- = \omega(\sqrt{\log N/N})$, as required for \eqref{eq-p--lower}
	
\subsection{Proof of technical lemmas}\label{High d technical proofs}
\begin{proof}[Proof of Lemma \ref{QV convergence}]
	By \cite[Theorem 4.3.1]{LawlerLimic10}, there exists a constant $c_d$ such that the following holds
	\[
	G(\0,x) = \frac{c_d}{|x|^{d-2}} + O\left(\frac{1}{|x|^d}\right).
	\]
	Since $G(x,y) = G(\0,y-x)$ for $x,y \in \mbb{Z}^d$, we can deduce that for any $v \in V_N$
	\[
	G(v,x_K) - G(\0,x_K) = G(\0,x_K) O\left(\frac{|v|}{K}\right) \quad \text{as} \quad K \to \infty.
	\]
	By this, we mean that the suppressed constant does not depend on $v$ or $K$. We note that by \eqref{Metric Green Interpolation} this bound extends to $v \in \ctG \cap [-N,N]^d$, and in particular to points $v$ on edges incident to $\0$ (i.e. such that $|v| < 1$). From this we obtain
	\begin{align*}
	\var[\til\phi_{x_K} \mid \mc{F}_{\mc I}]= & G(x_K,x_K) - \sum_{v \in \partial \mc I} \Hm(x_K,v;\mc I) G(v,x_K)\\
	= & G(x_K,x_K) - G(\0,x_K)\Hm(x_K,\mc I) \\&+ G(\0,x_K)\sum_{v \in \partial \mc I} \Hm(x_K,v;\mc I) O\left(\frac{|v|}{K}\right).
	\end{align*}
	Similarly, we have $\var[\til\phi_{x_K} \mid \mc{F}_\0] = G(x_K,x_K) - G(\0,x_K)\Hm(x_K,\0)$. Therefore, we see that it suffices to bound the last term above by the difference of the harmonic measures. Let $\tau_\0 = \inf\{t \geq 0 \,:\, \til{B}_t = \0\}$ be the hitting time of $\0$ by a metric graph Brownian motion. We have
	\begin{align*}
	\Hm(x_K,\mc I) - \Hm(x_K, \0) &= \sum_{v \in \partial \mc I} \Hm(x_K,v;\mc I)\mbb{P}_v(\tau_\0 = \infty) \\
	&= \sum_{v \in \partial \mc I} \Hm(x_K,v;\mc I)\Omega(|v| \wedge 1),
	\end{align*}
	where we have used the fact that there exists a constant $p_d > 0$ such that $\mbb{P}_v(\tau_\0 = \infty) \geq p_d$ for all $v \in \mbb{Z}^d \setminus \{\0\}$. This concludes the proof.
\end{proof}
\begin{proof}[Proof of Lemma \ref{Capacity bounds}]
	Both bounds are proved by similar arguments so we only provide the details for the bound on $f_2(N)$. First, note that $\Hm(x_K,\0)/G(\0,x_K) =  1/\sigma_d^2$ so it suffices to show $\Hm(x_K,\mc I)/G(\0,x_K) = \Omega(N/(\log N)^{\one_{d = 3}})$ uniformly over compact connected sets $\mc I$ containing $\0$ and intersecting $\partial V_N$ and $K$ large enough (say $K \geq 100 N$). Begin by noting that any such $\mc I$ contains a set $\mc U = \{u_j\}_{j = 0}^N$ where $u_j \in \mbb{Z}^d$ and $|u_j|_{\infty} = j$ so it suffices to lower bound $\Hm(x_K, \mc U)$. To this end, let $Y$ be the number of visits to $\mc U$ by a random walk started at $x_K$. We have 
	\[
	\Hm(x_K, \mc U) = \mbb{P}(Y > 0) = \frac{\E[Y]}{\E[Y\mid Y > 0]}.
	\]
	By \cite[Theorem 4.3.1]{LawlerLimic10}, we have 
	\[
	\E[Y] = \sum_{u \in \mc U} G(u,x_K) = \Omega\left(N\right) G(\0,x_K)
	\]
	On the other hand, letting $\tau_{\mc U} = \inf\{t \geq 0 \,:\, \til B_t \in \mc U\}$ be the hitting time of $\mc U$, the following holds uniformly over $u \in \mc U$
	\begin{align*}
	\E[Y \mid \til B_{\tau_{\mc U}} = u] &= \sum_{u' \in \mc U} G(u,u') = O\left(\sum_{j = 1}^N \frac{1}{j^{d-2}}\right) = O \left(\log(N)^{\one_{d = 3}}\right).
	\end{align*}
	Combining the two estimates gives the desired result.
\end{proof}

	\section{Chemical distance in two dimensions}\label{2D proofs}
	This section is devoted to the proof of Theorem~\ref{2D result}.  Recall that $\til{\phi}_N$ is the Gaussian free field on the metric graph of $V_N$ with Dirichlet boundary conditions, and that $G_N$ is Green's function on $V_N$ as in \eqref{Green definition}. The proof employs the same type of exploration martingale as in the case of $d \geq 3$. Below we prove Theorem~\ref{2D result} while postponing proofs of a few lemmas to later subsections.
	\begin{proof}[Proof of Theorem~\ref{2D result}]
		For $h \in \mathbb{R}$ and $0 < \alpha < \gamma < 1$, define $\mc{E}_{N,1} = \{ D_{N, h}(V_{\alpha N},\partial V_{\gamma N}) < \infty\}$. That is,  $\mc{E}_{N,1}$ is the event that $V_{\alpha N}$ is connected to $\partial V_{\gamma, N}$ in $\til{E}_{N}^{\geq h}$.
		\begin{lemma}\label{2D percolation probability}
			We have
			\[
			c_1 = \inf \{\PP(\mc{E}_{N,1}) \,:\, N \geq 1\} > 0\,,
			\]
			where $c_1$ depends on $h$, $\alpha$, and $\gamma$.
		\end{lemma}
		\begin{remark}
			Despite the fact that the statement of  Lemma~\ref{2D percolation probability} is formally slightly stronger than \cite[Proposition 4]{DingLi18} (since percolation on metric graph is a sub-event of percolation on discrete lattice), the proof of \cite[Proposition 4]{DingLi18} adapts with essentially no change. Thus, we omit further details of the proof.
		\end{remark}
		Now, let $\mu = (1+\gamma)/2$ and $M_{\mu N}$ be the exploration martingale with target set $\partial V_{\mu N}$ and source set $\cI_0 = V_{\alpha N}$, as defined in \eqref{eq-def-M}. That is to say,
		\begin{align*}
		X_{\mu N} = \frac{1}{|\partial V_{\mu N}|} \sum_{v \in \partial V_{\mu N}} \til{\phi}_{N,v} \mbox{ and }
		M_{\mu N,t} = \E[X_{\mu N} \mid \mc{F}_{\cI_t }]\,,
		\end{align*}
		where $\cI_t = \{v\in V_N: D_{N, h}(V_{\alpha N}, v)\leq t\}$. From now on, we take $N$ large enough that the boxes $V_{\alpha N}, V_{\beta N}, V_{\gamma N}, V_{\mu N}$ are distinct. For $t \geq 0$ we let $\partial \cI_t^+ = \partial \cI_t \cap \til{E}_{N}^{> h}$ be the points on $\partial \cI_t$ where $\til{\phi}$ is strictly above $h$, which we will refer to as the active points at time $t$ (by active here we mean that these are the points from which the metric ball exploration can proceed further), and let $\partial \cI_t^- = \partial \cI_t \setminus \til{E}_{N}^{\geq h}$ be the points on $\partial \cI_t$ where $\til{\phi}$ is strictly below $h$. We then define the ``positive'' and ``negative'' parts of $M_{\mu N}$ (which we denote by $M_{\mu N}^{\pm}$) as
		\[
		M_{\mu N,t}^{\pm} = \sum_{u \in \partial \cI_t^{\pm}} \Hm_{N,t}(\partial V_{\mu N},u)(\til{\phi}_u - h)\,,
		\]
		where $\Hm_{N,t}(v,u) = \Hm(v,u \,;\, \cI_t \cup \partial V_N)$. We note that $\partial \cI_t^- = \partial \cI_0^-$ for all $t$, which, combined with the fact $\mc I_t$ is increasing, implies $M_{\mu N}^-$ is increasing. For $c\in \mathbb R$, define 
		$$\mc{E}_{N,2}(c) = \{M_{\mu N,t}^+ \geq c \,\, \mbox{ for all } 0 \leq t \leq D_{N, h}(V_{\alpha N}, \partial V_{\beta N})\}\,.$$
		\begin{lemma}\label{Positive part lemma}
			There exists a constant $c_2 = c_2(h, \beta, \gamma)> 0$ such that
			\[
			\PP(\mc{E}_{N,2}(c_2 \epsilon) \mid \mc{E}_{N,1}) \geq 1 - \epsilon \mbox{  for all } \epsilon > 0 .
			\]
		\end{lemma}
		For the rest of the section we let $\mc{E}_{N,2} = \mc{E}_{N,2}(c_2 \epsilon/2)$ for convenience. The core idea in proving Theorem~\ref{2D result} is to bound from below the rate at which the quadratic variation increases as a function of $M_{\mu N,t}^+$. Combined with an upper bound on $\langle M_{\mu N}\rangle_{D_{N, h}(V_{\alpha N}, \partial V_{\beta N})}$, this then yields an upper bound on $D_{N, h}(V_{\alpha N}, \partial V_{\beta N})$. In order to carry out the proof,
		we first give the upper bound on the quadratic variation of $M_{\mu N}$ (which is  easier than the lower bound). 
		\begin{lemma}\label{Useful 2D Green bounds}\cite[Lemma 2]{DingLi18}
			For $0 < \mu < 1$, there exist constants $c, c' > 0$ such that
			\begin{align}
			&\sum_{v \in \partial V_{\mu N}} G_N (u,v) \leq c N, \quad \forall u \in \partial V_{\mu N}; \label{Time spent on the boundary} \\
			&G_N(u,v) \geq c', \quad \forall u,v \in V_{\mu N}. \label{Green lower bound}
			\end{align}
		\end{lemma}
		By \eqref{Time spent on the boundary} and Corollary~\ref{QV form}, we get that for some constant $c_3 = c_3(\mu) > 0$,
		\begin{equation}\label{2D total QV upper bound}
		\langle M_{\mu N}\rangle_{\infty} \leq \frac{1}{|\partial V_{\mu N}|^2} \sum_{v,v' \in \partial V_{\mu N}} G(v,v') \leq c_3.
		\end{equation}
		The remaining main task for proving Theorem~\ref{2D result} is to show that on some event  $\mathcal E^*_N$ with $\PP(\mathcal E^*_N \mid \mathcal E_{N,1}) \geq 1-\epsilon$, we have
		\begin{equation}\label{eq-remaining-main-task}
		\langle M_{\mu N}\rangle_{D_{N, h}(V_{\alpha N}, \partial V_{\beta N})} \geq \kappa \frac{D_{N, h}(V_{\alpha N}, \partial V_{\beta N})^2}{N^2 \sqrt{\log N}}\,,
		\end{equation}
		for some constant $\kappa= \kappa(\epsilon, \alpha, \beta, \gamma, h)>0$.
		Indeed, assuming \eqref{eq-remaining-main-task}, we can then combine it with \eqref{2D total QV upper bound}, and conclude that on $\mc{E}_N^*$
		\[
		D_{N, h}(V_{\alpha N}, \partial V_{\beta N}) \leq c_3 \kappa^{-1/2} N (\log N)^{\frac{1}{4}}. 
		\]
		completing the proof of Theorem~\ref{2D result}.
		
		\medskip
		
		It remains to show \eqref{eq-remaining-main-task}. To this end, we first bound the quadratic variation from below in terms of the $\ell_2$-norm of the harmonic measure on the active points, as in the next lemma.
		\begin{lemma}\label{QV by harmonic measure L2}
			There exists a constant $c_4 > 0$ such that the following holds almost surely for all integers $K \geq 1$, 
			\[
			\langle M_{\mu N}\rangle_K  \geq c_4 \sum_{k = 1}^K \sum_{u \in \partial \cI_k^+} \Hm_{N,k}(\partial V_{\mu N},u)^2.
			\]
		\end{lemma}
		Let $\mc{I}^+ = \cup_{k = 1}^{D_{N,h}(V_{\alpha N}, \partial V_{\beta N})} \partial \mc{I}^+_k$.  Since the  sets $\{\partial \mc{I}^+_k\,:\, k \geq 1\}$ are disjoint, for  $u\in \mc{I}^+$ we can define
		\[
		W(u) = \Hm_{N,k}(\partial V_{\mu N},u),
		\]
		where $k$ is the unique positive integer such that $u \in \partial \mc{I}^+_k$. We rewrite the conclusion of Lemma~\ref{QV by harmonic measure L2} as
		\begin{equation}\label{2D QV lower bound}
		\langle M_{\mu N} \rangle_{D_N(V_{\alpha N}, \partial V_{\beta N})} \geq c_4\sum_{u \in \mc{I}^+} W(u)^2.
		\end{equation}
		In order to bound the right-hand side of \eqref{2D QV lower bound} from below by $M_{\mu N}^+$, we need some control on the empirical profile of $\{\til{\phi}_{N, v}: v\in \mc{I}^+\}$. To this end, we define
		\begin{align*}
		\mc{B}_0 &= \{u \in \mc{I}^+ \,:\, \til{\phi}_{N,u} - h \leq \sqrt{\log N}\},\\
		\mc{B}_j &= \{u \in \mc{I}^+ \,:\, 2^{j-1} \sqrt{\log N} < \til{\phi}_{N,u} - h \leq 2^j \sqrt{\log N}\}, \quad j \geq 1.
		\end{align*}
		Here the scale $\sqrt{\log N}$ is chosen to match the order of $\sqrt{\E[\til{\phi}_u^2]}$ for $u \in V_{\beta N}$. Letting $\mathcal W_j = \sum_{u \in \mc{B}_j} W(u)$, we get from the Cauchy-Schwartz inequality that
		\begin{equation}\label{eq-L2-bound-W}
		\sum_{u \in \mc{I}^+} W(u)^2 \geq \sum_{j = 0}^\infty \frac{\mathcal W_j^2}{|\mc{B}_j|}\,,
		\end{equation}
		where we use the convention $0/0 = 0$. The appearance of $|\mc{B}_j|$ in the denominator in the preceding inequality calls for an upper bound on $|\mc{B}_j|$, as incorporated in the next lemma (the reason for the specific form of the bound will be made clear below). 
		\begin{lemma}\label{Size of extreme connected component}
			Let $f$ be a function on the positive integers such that $f(j) = e^{O(j)}$. Then there exists $c_5 = c_5(\alpha,\beta,h, f)> 0$ such that
			\[
			\sup \{f(j) \E[|\mc{B}_j|] \,:\, j \geq 0\} \leq c_5\frac{N^2}{\sqrt{\log N}}\,.
			\]
		\end{lemma}
		We are now ready to give a lower bound on $\langle M_{\mu N}\rangle_{D_{N, h}(V_{\alpha N}, \partial V_{\beta N})}$. By definition of $\mathcal W_j$, we have
		\[
		\mathcal W_j \geq 2^{-j} (\log N)^{-\frac{1}{2}} \sum_{u \in \mc{B}_j} W(u) (\til{\phi}_{N,u} - h).
		\]
		In addition, on the event $\mc{E}_{N,1} \cap \mc{E}_{N,2}$ we have
		\begin{align*}
		\sum_{j = 1}^\infty \sum_{u \in \mc{B}_j}W(u) (\til{\phi}_{N,u} - h) &= \sum_{k = 1}^{D_{N, h}(V_{\alpha N}, \partial V_{\beta N})} M_{\mu N, k}^+ \\
		&\geq \frac{c_2\epsilon D_{N, h}(V_{\alpha N}, \partial V_{\beta N}) }{2}\,.
		\end{align*}
		Letting $c_6 = 6/\pi^2$ so that $c_6 \sum_{j = 0}^\infty (j+1)^{-2} = 1$, we see that
		\begin{equation}\label{eq-lower-bound-W-j}
		\mc{E}_{N,1} \cap \mc{E}_{N,2} \subseteq \bigcup_{j=0}^\infty \left\{\mathcal W_{j} \geq \frac{c_6 c_2\epsilon D_{N, h}(V_{\alpha N}, \partial V_{\beta N})}{2^{j+1}(j+1)^2\sqrt{\log N}}\right\}\,.
		\end{equation}
		Letting $c_7 = 20/(c_1\epsilon)$, we define $\mc{E}_{N,3} = \bigcap_{j = 0}^\infty \mc{E}_{N,3,j}$, where
		\begin{align*}
		\mc{E}_{N,3,j} = \left\{|\mc{B}_j| \leq c_7 \E[|\mc{B}_j|] (1.1)^{j+1} \right\}, \quad j \geq 0. 
		\end{align*}
		By Markov's inequality, we get that 
		\begin{equation}\label{eq-E-3-prob}
		\PP(\mc{E}_{N,3}^c \mid \mc{E}_{N,1}) \leq \frac{\mbb{P}(\mc{E}_{N,3}^c)}{c_1} \leq \frac{\epsilon}{2}\,.
		\end{equation}
		Let $\mc{E}_N^* = \mc{E}_{N,1} \cap \mc{E}_{N,2} \cap \mc{E}_{N,3}$. By Lemma~\ref{Positive part lemma} and \eqref{eq-E-3-prob}, we get that
		\begin{equation}\label{eq-mathcal-E*-prob}
		\PP(\mc{E}_N^* \mid \mc{E}_{N,1}) \geq 1 - \epsilon\,.
		\end{equation}
		We deduce from \eqref{2D QV lower bound}, \eqref{eq-L2-bound-W} and \eqref{eq-lower-bound-W-j} that on $\mc{E}_N^*$
		\[
		\langle M_{\mu N}\rangle_{D_{N, h}(V_{\alpha N}, \partial V_{\beta N})} \geq \frac{c_6^2c_4c_2^2 c_1\epsilon^3}{20} \inf_{j\geq 0}\frac{D_{N, h}(V_{\alpha N}, \partial V_{\beta N})^2}{(4.4)^{j+1}(j+1)^4\E[|\mc{B}_{j}|]\log N}. 
		\]
		Combined with Lemma~\ref{Size of extreme connected component}, this gives that on $\mathcal E^*_N$
		\[
		\langle M_{\mu N}\rangle_{D_{N, h}(V_{\alpha N}, \partial V_{\beta N})} \geq \frac{c_6^2c_4c_2^2 c_1\epsilon^3}{20 c_5}\frac{D_{N, h}(V_{\alpha N}, \partial V_{\beta N})^2}{N^2 \sqrt{\log N}}.
		\]
		Combining with \eqref{eq-mathcal-E*-prob}, we have completed the verification of \eqref{eq-remaining-main-task} as promised.
	\end{proof}

	\subsection{Proof of Lemma \ref{Positive part lemma}} \label{Positive part pf}
	
	We first give the main intuition behind the proof of Lemma \ref{Positive part lemma} in the case when $h = 0$. On the event $\mc{E}_{N,1} \cap \mc{E}_{N,2}(\epsilon)^c$, we have $M_{\mu N,s} \leq \epsilon + M_{\mu N,s}^-$ for some $s \leq D_{N,h}(V_{\alpha N},\partial V_{\beta N})$. However, we also have $M_{\mu N,t} \geq M_{\mu N,s}^-$ for all $t \geq s$. Since  $D_{N,h}(V_{\alpha N},\partial V_{\gamma N}) < \infty$ on $\mc{E}_{N,1}$, the martingale must stay above $M_{\mu N,s}^-$ after time $s$ and yet accumulate an order 1 amount of quadratic variation --- this happens with small probability. The case for general $h$ is similar but a bit more complicated. We carry out a detailed proof below. 
	
	In this subsection and the ones that follow, we let $c > 0$ be an arbitrary constant whose value may change each time it appears, and which may depend on $h, \alpha, \beta, \gamma$ but not on $N$. Recall the proces $\pi_{\mu N,t} = \Hm(\partial V_{\mu N}, \mc I_t)$ and the fact that  $M_{\mu N, t}^-$ is  increasing in $t$. From this we obtain that for $0 \leq s < t < D_{N, h}(V_{\alpha N},\partial V_{\mu N})$
	\begin{align*}
	M_{\mu N, t} - M_{\mu N, s} &\geq h[\pi_{\mu N,t} - \pi_{\mu N, s}] -M_{\mu N, s}^+\,,\end{align*}
	with equality if and only if $M_{\mu N,t}^+ = M_{\mu N,s}^+ = 0$. Next, we claim that there exist constants $c,c'$ such that for any $0\leq s < t < D_{N, h}(V_{\alpha N}, \partial V_{\gamma N})$,
	\begin{align*}
	c' [\pi_{\mu N,t} - \pi_{\mu N,s}] &\leq \langle M_{\mu N}\rangle_t - \langle M_{\mu N}\rangle_s, \\
	\langle M_{\mu N}\rangle_t - \langle M_{\mu N}\rangle_s &\leq c[\pi_{\mu N,t} - \pi_{\mu N,s}].
	\end{align*}
	To see this, note that we have
	\begin{align*}
	\langle M_{\mu N}\rangle_t - \langle M_{\mu N}\rangle_s &= \frac{1}{|\partial V_{\mu N}|}\sum_{w \in \partial \mc I_t} \Hm_{N,t}(\partial V_{\mu N}, w)G_{N,s}(w,\partial V_{\mu N}),\\
	\pi_{\mu N,t} - \pi_{\mu N,s} &= \sum_{w \in \partial \mc I_t} \Hm_{N,t}(\partial V_{\mu N}, w)\mbb{P}_w(\tau_{\partial V_N} < \tau_{\mc I_s}),
	\end{align*}
	where $G_{N,s}$ is the Green's function on $V_N \setminus \mc I_s$, $G_{N,s}(w, \partial V_{\mu N}) = \sum_{v \in \partial V_{\mu N}} G(w,v)$, and $\tau_A$ is as usual the hitting time of $A$ by a metric graph Brownian motion. By Lemma~\ref{Useful 2D Green bounds} and the assumption that $\mc I_s \subset V_{\gamma N}$,
	\[
	c' \mbb{P}_w(\tau_{\partial V_{\mu N}} < \tau_{\mc I_s}) \leq \frac{G_{N,s}(w, \partial V_{\mu N})}{|\partial V_{\mu N}|} \leq c \mbb{P}_w(\tau_{\partial V_{\mu N}} < \tau_{\mc I_s}).
	\]
	Similarly, we have by straightforward random walk considerations (namely the invariance principle) 
	\[
	c' \mbb{P}_w(\tau_{\partial V_{\mu N}} < \tau_{\mc I_s}) \leq \mbb{P}_w(\tau_{\partial V_N} < \tau_{\mc I_s}) \leq \mbb{P}_w(\tau_{\partial V_{\mu N}} < \tau_{\mc I_s}).
	\]
	Thus, the upper and lower bounds on the quadratic variation follow. Altogether, this implies
	\begin{equation}\label{eq-M-difference-lower-bound}
	M_{\mu N, t} - M_{\mu N, s} > c h [\langle M_{\mu N}\rangle_t - \langle M_{\mu N}\rangle _s] - M_{\mu N, s}^+\,,
	\end{equation}
	for all $0\leq s < t \leq D_{N,h}(V_{\alpha N}, \partial V_{\gamma N})$ such that $M_{\mu N,s}^+ > 0$.
	For any $x > 0$, let $\eta_x = \inf \{ t \,:\, M_{\mu N,t}^+ \leq x \}$ and define the martingale $\til{M}^x_{\mu N}$ (with respect to $\mc{G}_t = \mc{F}_{\cI_{\eta_x + t}}$) by
	\[
	\til{M}^x_{\mu N,t} = \begin{cases}
	M_{\mu N, \eta_x + t} - M_{\mu N, \eta_x} &\eta_x < \infty,\\
	0 &\eta_x = \infty.
	\end{cases}
	\]
	Let $\Delta = D_{N, h}(V_{\alpha N}, \partial V_{\gamma N}) - \eta_x$ and note that on $\mc{E}_{N,1} \cap \mc{E}_{N,2}(x)^c$ we get from \eqref{eq-M-difference-lower-bound} that
	\[
	\til{M}_{\mu N,t}^x > ch \langle \til{M}_{\mu N}^x\rangle_t - x, \quad 0\leq t < \Delta.
	\]
	Using the lower bound on the quadratic variation by the Harmonic measure that was proven above, we see that the following bound holds almost surely on $\mc{E}_{N,1}\cap \mc{E}_{N,2}(x)^c$,
	\begin{align*}\label{QV change constant lower bound}
	\langle \til{M}_{\mu N}^x\rangle_\Delta &\geq \langle M_{\mu N}\rangle_{D_{N, h}(V_{\alpha N}, \partial V_{\gamma N})} - \langle M_{\mu N}\rangle_{D_{N, h}(V_{\alpha N}, \partial V_{\beta N})}\\
	&\geq c' [\pi_{\mu N, D_{N,h}(V_{\alpha N},\partial V_{\gamma N})} - \pi_{\mu N, D_{N,h}(V_{\alpha N},\partial V_{\beta N})}] \geq c',
	\end{align*}
	where the last inequality follows from a simple adaptation of the proof of \cite[Proposition 4]{DingLi18}.
	
	Write $T_t = \inf\{s \,:\, \langle \tilde M^x_{\mu N} \rangle_s > t\}$ and let $B$ be a standard Brownian motion that satisfies $B_t = \tilde M^x_{\mu N, T_t} - \tilde M^x_{\mu N,0}$ for $t < \langle \tilde M^x_{\mu N} \rangle_\infty$ and is independent of $\mc{F}_{\mc I_{\eta_x}}$. Letting $\tau_{h,x} = \inf\{ t \,:\, B_t \leq ch t - x\}$, it follows from Proposition \ref{Hitting times with drift} and Lemma \ref{Elementary hitting time bound} that for some $c''$
	\[
	\PP(\mc{E}_{N,1} \cap \mc{E}_{N,2}(x)^c) \leq \PP(\tau_{h,x} \geq c') \leq c'' x.
	\]
	Since $\mbb{P}(\mc{E}_{N,1})$ is bounded away from $0$ by Lemma \ref{2D percolation probability}, the conclusion follows.

	\subsection{Proof of Lemma \ref{QV by harmonic measure L2}}\label{QV by L2 pf}
	
	Define $d_k = \langle M_{\mu N} \rangle_k - \langle M_{\mu N} \rangle_{k-1}$ (throughout this section $k$ and $K$ are positive integers), and $\mc{A}_k = \partial \cI_k \setminus \cI_{k-1}$. Let $\cI^{+, K} = \cup_{k=1}^K \partial \cI_k^+$, $\mc{A} = \cup_{k=1}^K \mc{A}_k$, and note that for all $k \geq 0$, $\partial \cI_k^+ \subset \mc{A}_k \cap V$. By Corollary \ref{QV form}, we have
	\begin{align*}
	d_k &= \frac{1}{|\partial V_{\mu N}|}\sum_{v \in \partial V_{\mu N}}\sum_{u \in \mc{A}_k} \Hm_{N,k}(\partial V_{\mu N}, u) G_{N,k-1}(u,v) \\
	&\geq\sum_{u \in \mc{A}_k}(\Hm_{N,k}(\partial V_{\mu N}, u))^2 G_{N,k-1}(u,u)\,,
	\end{align*}
	where the inequality follows from $G_{N, k-1}(u, v) \geq \Hm_{N,k}(v, u) G_{N, k-1}(u, u)$. Consequently,
	\begin{equation}\label{eq-variation-decomposition}
	\langle M_{\mu N}\rangle_{K} \geq \sum_{k = 1}^K \sum_{u \in \mc{A}_k}(\Hm_{N,k}(\partial V_{\mu N}, u))^2 G_{N,k-1}(u,u).
	\end{equation}
	Comparing \eqref{eq-variation-decomposition} to the desired inequality in Lemma~\ref{QV by harmonic measure L2}, we see two differences: (1) the summation in \eqref{eq-variation-decomposition} is over $\mc{A}_k$ as opposed to $\partial \mc{I}_k^+$; (2) there is a term $G_{N,k-1}(u,u)$ in \eqref{eq-variation-decomposition} which we need to bound from below. To address this, we will define a function $\psi : \cI^{+, K} \to \mc{A}$ which, roughly speaking,  allows us to bound $(\Hm_{N,k}(\partial V_{\mu N}, u))^2 G_{N,k-1}(u,u)$ from below by $(\Hm_{N,\tau}(\partial V_{\mu N}, \psi(u)))^2$ (where $\psi(u) \in \mc{A}_{\tau}$). We next carry out the details.
	
	To specify $\psi$, let $D_{\ell_1}$ be $\ell_1$-distance on $\mathbb R^2$ (up to scaling, this is the graph distance on $\til{\mc G}$), and $u \in \partial \cI_k^+ \subseteq \mc{A}_k$  be an active point at time $k$.  If $D_{\ell_1}(u, \cI_{k-1}) \geq 1/2$, then $G_{N,k-1}(u,u) \geq 1/2$ (see \cite{Lupu16}) and we let $\psi(u) = u$. If, on the other hand, $D_{\ell_1}(u, \cI_{k-1}) < 1/2$, there exist at most four points on $\partial \cI_{k-1} \cap B_{\ell_1}(u,1/2)$ (here $B_{\ell_1}(u,r)$ denotes the open ball of radius $r$ centered at $u$ with respect to $D_{\ell_1}$). For every $w \in \partial \cI_{k-1} \cap B_{\ell_1}(u,1/2)$ there is a unique (random) integer $\tau_w  \leq k-1$ such that $w \in \mc{A}_{\tau_w}$. We let $\psi(u)= w$ be the point in $\partial \cI_{k-1} \cap B_{\ell_1}(u,1/2)$ that minimizes $\tau_w$ (that is, the ``oldest'' $w$), breaking ties by distance to $u$ (choosing the $w$ closest to $u$). With this choice, $ \partial \cI_{\tau_w} \cap B_{\ell_1}(u, |u-w|_{\ell_1}) = \emptyset$ so $\Hm_{N,\tau_w}(u,w) \geq 1/4$ and hence 
	\begin{equation}\label{eq-compare-hm}
	\Hm_{N,\tau_w}(\partial V_{\mu N},w) \geq \frac{1}{4} \Hm_{N,k}(\partial V_{\mu N}, u).
	\end{equation}
	Also, $D_{\ell_1}(u, \cI_{\tau_w-1}) \geq 1/2$ so by \eqref{Metric Green Interpolation} we have for $\delta = D_{\ell_1}(w,u) < 1/2$
	\begin{equation}\label{eq-compare-Green}
	G_{N,\tau_w-1}(w,w) = 4\delta(1-\delta) + (1-\delta)^2 G_{N,\tau_w-1}(u,u) \geq \frac{1}{2}.
	\end{equation}
	Finally, for distinct $u,u' \in \cI^{+, K}$, $B_{\ell_1}(u,1/2) \cap B_{\ell_1}(u',1/2) = \emptyset$ so $\psi$ is injective. Recalling \eqref{eq-variation-decomposition}, we get that
	\begin{align*}
	\langle M_{\mu N}\rangle_{K} &\geq  \sum_{k = 1}^K \sum_{u \in \partial \mc{I}_k^+}(\Hm_{N, \tau_{\psi(u)}}(\partial V_{\mu N}, \psi(u)))^2 G_{N,\tau_{\psi(u)}-1}(\psi(u), \psi(u)) \\
	&\geq \frac{1}{32}\sum_{k = 1}^K \sum_{u \in \partial \mc{I}_k^+}\Hm_{N,k}(\partial V_{\mu N}, u)^2\,,
	\end{align*}
	where the factor of $\frac{1}{32}$ comes from $(\frac14)^2$ (which accounts for the ratio on the square of harmonic measures; see \eqref{eq-compare-hm}) and $\frac{1}{2}$ (which accounts for the Green function term; see \eqref{eq-compare-Green}).

	\subsection{Proof of Lemma \ref{Size of extreme connected component}}\label{Extreme component pf}

	Let $A_N = V_{\beta N} \setminus V_{\alpha N}$. We will bound $\E[|\mc{B}_j|]$ by bounding the probability that each vertex $v \in A_N$ belongs to $\mc{B}_j$. Note that 
	\[
	\mc{I}^+ \subseteq \{v \in A_N,\, D_{N, h}(V_{\alpha N}, v) < \infty\}.
	\]
	By Theorem~\ref{Gibbs-Markov}, $\PP(D_{N, h, \beta}(V_{\alpha N},v) < \infty \mid \til{\phi}_{N,v})$ is increasing in $\til{\phi}_{N,v}$. Therefore letting $a_j = h + 2^{j-1}\sqrt{\log N}\one_{j > 0}$ and $b_j = h + 2^j \sqrt{\log N}$ we have for $j\geq 0$ and $v \in A_N$
	\begin{equation}\label{Extreme value percolation decomposition}
	\PP(v \in \mc{B}_j) \leq \PP(\til{\phi}_{N,v} > a_j) \PP(D_{N, h}(V_{\alpha N},v) < \infty \mid \til{\phi}_{N,v} = b_j).
	\end{equation}
	Since (see, e.g., \cite[Theorem 4.4.4, Proposition 4.6.2]{LawlerLimic10})
	\begin{equation}\label{2D variance asymptotic}
	\var[\til{\phi}_{N,v}] = G_N(v,v) =  \frac{2}{\pi} \log N + O(1),
	\end{equation} 
	we see that there exists a constant $c = c(h,\beta)> 0$ such that for all $j \geq 0$ and $v \in A_N$
	\begin{equation}\label{Extreme value probability bound}
	\PP(\til{\phi}_{N,v} > a_j) \leq e^{-c4^j}.
	\end{equation}
	We will bound the second term of \eqref{Extreme value percolation decomposition} in terms of $k$ for $v\in \partial V_{\alpha N + k}$. We state the result here and defer the proof to the end of this section.
	\begin{lemma}\label{2D connection probability bound}
		Let $k^* = (\beta - \alpha) N /\sqrt{\log N}$. There exists a positive constant $c = c(h,\beta, \alpha)$ such that for all $j \geq 0$, $ k \geq k^*$ and $v \in \partial V_{\alpha N + k}$,
		\[
		\mbb{P}(D_{N, h}(V_{\alpha N},v) < \infty \mid \til{\phi}_{N,v} = b_j) \leq \frac{c 2^j}{\sqrt{\log N}} \sqrt{\log N - \log k}.
		\]
	\end{lemma}
	Using Lemma~\ref{2D connection probability bound} and \eqref{Extreme value probability bound} (and the fact that $\E[|\mc{B}_j|] = \sum_{v \in A_N} \mbb{P}(v \in \mc{B}_j)$), we have
	\[
	\E[|\mc{B}_j|] \leq c 2^j e^{-c' 4^j}\left[\frac{N^2}{\sqrt{\log N}} + \frac{N}{\sqrt{\log N}} \sum_{k = k^*}^{(\beta-\alpha)N}\sqrt{\log N - \log k}\right],
	\]
	where we used the fact that there are $O(N^2/\sqrt{\log N})$ lattice points in $V_{\alpha N + k^*} \setminus V_{\alpha N}$. Finally, we note
	\begin{align*}
	\sum_{k = k^*}^{(\beta - \alpha)N} \sqrt{\log N - \log k}
	&= O(1) \int_{k^*}^{(\beta-\alpha)N} \sqrt{\log N - \log x}\, dx \\
	&= O(N) \int_{(\beta-\alpha)^{-1}}^{(\beta-\alpha)^{-1}\log N} \frac{\sqrt{\log u}}{u^2}\, du= O(N)\,.
	\end{align*}
	This completes the proof of Lemma \ref{Size of extreme connected component}.
	\begin{proof}[Proof of Lemma \ref{2D connection probability bound}] 
		As usual, the proof consists of analyzing an exploration martingale. To specify the martingale, we let $A_{N,v} = \partial (V_{(1-\gamma) N} + v)$ be the boundary of a box of radius $(1-\gamma) N$ around $v$. Note that $A_{N,v} \subset V_N$ for all $v \in V_{\beta N}$.  We then take $M_{N,v}$ to be the exploration martingale with source $\cI_0 = \{v\}$ and observable $X_A$ on $\til{E}_{N}^{\geq h}$. We also let $r_{N,v} = k \wedge((1-\gamma) N/2)$ and $W_{N,v} = \partial(V_{r_{N,v}} + v)$, and note that
		\[
		\{D_{N,h}(v, V_{\alpha N}) < \infty \} \subset \{D_{N,h}(v, W_{N,v}) < \infty \}.
		\]
		As before, we let $\Hm_{N,t}(u,w) = \Hm(u,w; \mc{I}_t \cup \partial V_N)$, $\Hm_N(u,\mc I_t) = \mbb{P}_u(\tau_{\mc I_t} < \tau_{\partial V_N})$, and $\pi_{N,v,t} = \Hm(A_{N,v}, \mc I_t)$. We have
		\[
		M_{N,v,t} - M_{N,v,0} \geq h[\pi_{N,v,t} - \pi_{N,v,0}] - 2^j \sqrt{\log N}\pi_{N,v,0},
		\]
		with equality if and only if $\pi_{N,v,t} = \pi_{N,v,\infty}$ (recall we assume $\til\phi_v = 2^j\sqrt{\log N} + h$). In particular, on $\{D_{N,h}(v, W_{N,v}) < \infty \}$ we have strict inequality for $0 \leq t < D_{N,h}(v, W_{N,v})$. Arguing as in the proof of Lemma~\ref{Positive part lemma}, we can show that there exist constants $c_-,c_+ > 0$ such that the following holds for all $0 \leq s \leq t \leq D_{N,h}(v, W_{N,v})$,
		\[
		c_-[\pi_{N,v,t} - \pi_{N,v,s}] \leq \langle M_{N,v}\rangle_t - \langle M_{N,v}\rangle_s \leq c_+[\pi_{N,v,t} - \pi_{N,v,s}].
		\]
		Additionally, it follows from Lemma \ref{Useful 2D Green bounds} and \eqref{2D variance asymptotic} that there exists $c_1$ independent of $k$ such that $\pi_{N,v,0} \leq c_1/\log N$. All together, this shows that the following holds almost surely on $\{D_{N,h}(v, W_{N,v}) < \infty \}$,
		\[
		M_{N,v,t} - M_{N,v,0} \geq c_2h\langle M_{N,v}\rangle_t - \frac{c_1 2^j }{\sqrt{\log N}}, \quad 0 \leq t < D_{N,h}(v, W_{N,v}).
		\]
To simplify notation, we let $m = c_2h$ and $b = c_1 2^j/\sqrt{\log N}$. As usual, we write $T_t = \inf\{s \,:\, \langle M_{N,v}\rangle_s > t\}$ and let $B$ be a standard Brownian motion independent of $\til\phi_v$ such that $B_t = M_{N,v,T_t} - M_{N,v,0}$ for $t < \langle M_{N,v}\rangle_\infty$. Then we can define
		\begin{align*}
		\tau_{h,N} &= \inf\left\{t \,:\, B_t \leq mt - b\right\}, \\
		\pi_{N,v}^- &= \inf_{\cI}\{\Hm_N(A_{N,v}, \mc I)\} - \Hm_N(A_{N,v}, v),
		\end{align*}
		where the infimum is taken over all closed, connected subsets of $\ctG_{N}$ containing $v$ and intersecting $W_{N,v}$. For notational convenience, we write $T= c_- \pi_{N,v}^-$. We have that
		\[
		\mbb{P}(D_{N,h}(v, W_{N,v}) < \infty) \leq \mbb{P}(\tau_{h,N} \geq T).
		\]
		Note that the right-hand side of the inequality is decreasing in $h$, so we assume $h \leq 0$. Applying Proposition \ref{Hitting times with drift} and Lemma \ref{Elementary hitting time bound}, we get that
		\begin{align*}
		\mbb{P}(\tau_{h,N} \geq T) &\leq 2g\left(m\sqrt{T}, 0\right)\left[ \Phi\left(m\sqrt{T}\right) - \Phi\left(m\sqrt{T} - \frac{b}{\sqrt{T}}\right)\right]\\
		&\leq 2[\varphi(m\sqrt{T}) -m\sqrt{T}] \frac{b}{\sqrt{T}}.
		\end{align*}
		To conclude the proof we need the following bound on $\pi_{N,v}^-$
		\begin{equation}\label{2D harmonic measure lower bound}
		\pi_{N,v}^- \geq c (\log N - \log k )^{-1}.
		\end{equation}
		Provided with \eqref{2D harmonic measure lower bound} and noting the trivial bound $T \leq c_-$, we have $-m\sqrt{T} = O(1)$ and $b = O(1/\sqrt{\log N})$ which implies
		\[
		\mbb{P}(\tau_{h,N} \geq T) = O\left(\frac{2^j \sqrt{\log N - \log k }}{\sqrt{\log N}}\right),
		\]
		as claimed.

		It remains to prove \eqref{2D harmonic measure lower bound}. Note that it suffices to prove
		\[
		\pi_{N,v}^- \geq c (\log N - \log r_{N,v} )^{-1}.
		\]
		To this end, let $u$ be a point on $\mathcal I$ such that $r_{N,v}/2 \leq |u-v|_{\ell_\infty} < 1 + r_{N,v}/2$, and let $B_u, B'_u, B''_u$ be  boxes centered at $u$ of side length $r_{N,v}/4$, $r_{N,v}/8$, $r_{N,v}/16$ respectively. By \cite[Proposition 6.4.1]{LawlerLimic10}, we get that there exists $c' = c'(\alpha,\beta,\gamma) > 0$ such that for any $x\in A_{N,v}$, 
		$$\Hm_N(x, B'_u) \geq c'(\log N - \log r_{N,v})^{-1}\,.$$
		It is also obvious that once the random walk arrives at $\partial B'_u$, there is a probability bounded uniformly from below that the random walk range before exiting $B_u$ will contain a contour in $B_u\setminus B''_u$. In this case, the random walk will hit at least one point in $\mathcal I \cap B_u$. Therefore, we get that
		$$\Hm_N(x, \mathcal I\cap B_u) \geq c''(\log N - \log k)^{-1}\,,$$
		where $c''>0$ depends on $c'$. In addition, for any $w\in \mathcal I \cap B_u$, we have 
		$$\mathrm{Hm}(w, \partial V_N; \partial V_N \cup \{v\}) \geq 1/2$$
		(see e.g., \cite[Theorem 4.4.4., Proposition 4.6.2]{LawlerLimic10}). Altogether, this means that for any $x\in A_{N,v}$, we have 
		\begin{align*}
		\Delta \pi(x, \mathcal I) &:= \PP_x(\mbox{ random walk hits } \mathcal I \mbox{ but not } v \mbox{ before  it hits } \partial V_N) \\
		&\geq c(\log N - \log r_{N,v})^{-1}
		\end{align*}
		for a constant $c>0$. Noting that 
		\[
		\Hm_N(A_{N,v},\cI) - \Hm_N(A_{N,v}, v) =  \frac{1}{|A_{N,v}|} \sum_{x\in A_{N,v}} \Delta \pi(x, \mathcal I),
		\]
		this completes the verification of \eqref{2D harmonic measure lower bound}.
	\end{proof}

\end{document}